\documentclass[pdf, 10pt]{amsart}

\usepackage{amsthm, amsmath, amssymb}
\usepackage[OT1]{fontenc}
\usepackage[colorlinks,citecolor=blue,urlcolor=blue]{hyperref}
\usepackage{color}

\newtheorem{thm}[equation]{Theorem}

\newtheorem{defn}[equation]{Definition}

\newtheorem{lem}[equation]{Lemma}

\def\R{{\mathbb{R}}}
\def\T{{\mathbb{T}}}
\def\C{{\mathbb{C}}}
\def\N{{\mathbb{N}}}
\def\Z{{\mathbb{Z}}}

\renewcommand{\a}{\alpha}
\renewcommand{\b}{\beta}
\newcommand{\g}{\gamma}

\renewcommand{\d}{\delta}

\newcommand{\la}{\lambda}

\newcommand{\e}{\varepsilon}

\renewcommand{\t}{\tau}

\newcommand{\te}{\theta}
\newcommand{\s}{\sigma}

\newcommand{\vp}{\varphi}

\newcommand{\8}{\infty}

\newcommand{\vt}{\vartheta}
\newcommand{\vr}{\varrho}

 \numberwithin{equation}{section}

\begin{document}

\title[] {Roth's Theorem in the Piatetski--Shapiro primes}
\author[M.Mirek]
{Mariusz Mirek}
\address{M.Mirek \\
Universit\"{a}t Bonn \\
Mathematical Institute\\
Endenicher Allee 60\\
D--53115 Bonn \\
Germany}
 \email{mirek@math.uni-bonn.de}

\thanks{
The author was partially supported by NCN grant DEC--2012/05/D/ST1/00053}

\maketitle

\begin{abstract}
Let $\mathbf{P}$ denote the set of prime numbers and, for an appropriate function $h$, define a set $\mathbf{P}_{h}=\{p\in\mathbf{P}: \exists_{n\in\mathbb{N}}\ p=\lfloor h(n)\rfloor\}$. The aim of this paper is to show that every subset of $\mathbf{P}_{h}$ having positive relative upper density contains a nontrivial three--term arithmetic progression. In particular the set of Piatetski--Shapiro primes of fixed type $71/72<\gamma<1$, i.e.
$\{p\in\mathbf{P}: \exists_{n\in\mathbb{N}}\ p=\lfloor n^{1/\gamma}\rfloor\}$ has this feature. We show this by proving the counterpart of Bourgain--Green's restriction theorem for the set $\mathbf{P}_{h}$.
\end{abstract}

\section{Introduction and statement of results}
Let $A$ be a subset of positive integers, for any $N\in\N$ we define the density $\triangle_{A}(N)$ of $A$ to be the number $\triangle_{A}(N)=\frac{1}{N}|A\cap[1, N]|$, and then we define the upper density of $A$ to be the quantity $\bar{\triangle}(A)=\limsup_{N\to\8}\triangle_{A}(N)$. We will say that $A$ contains three--term arithmetic progression if there is $a\in A$ and $d\not=0$ such that $a, a+d, a+2d\in A$. Let $N\in\N$, then $r_3(N)$ denotes the Erd\"{o}s--Tur\'{a}n constant, which is the density of the largest set $A\subseteq\{1, 2,\ldots, N\}$ containing no non--trivial three--term arithmetic progression.

Before we formulate our results we begin with a sketch of the historical background, which will justify our motivations. On the one hand, in 1953 Roth \cite{Rot}  proved that any subset of $\N$ having positive upper density contains infinitely many non--trivial three--term arithmetic progressions. In particular, thanks to this remarkable result we know much more. Namely, that $r_3(N)=O((\log \log N)^{-1})$. After that there was no development until Heath--Brown \cite{HB1} and Szemer\'{e}di  \cite{Ser}. They showed that $r_3(N)=O((\log N)^{-c})$ for some small $c>0$. The next advance was done by Bourgain, who proposed a new approach based on analysis of Bohr sets, instead of passing to short subprogressions and obtained $r_3(N)=O((\log \log N)^{1/2}(\log N)^{-1/2})$ in \cite{B1}, and almost a decade later in \cite{B2} showed that $r_3(N)=O((\log \log N)^{2}(\log N)^{-2/3})$. Not long afterwards, Sanders \cite{San1} refined Bourgain's arguments \cite{B2} and proved that $r_3(N)=O((\log N)^{-3/4+o(1)})$. The best currently known result in this field also belongs to Sanders \cite{San2} and gives $r_3(N)=O((\log \log N)^{5}(\log N)^{-1})$.
It is worth mentioning that the methods of \cite{San2} are largely unrelated to these last achievements.

On the other hand, the same kind of questions (about the existence of non--trivial three--term arithmetic progressions) may concern subsets of integers with vanishing upper density. The set of the prime numbers $\mathbf{P}$ turned out to be a natural candidate to study, especially in view of the Van der Corput theorem \cite{vCor}, where it was established that the set $\mathbf{P}$ contains infinitely many arithmetic progressions of length three. Not long ago, we waited until a common generalization of the theorem of Roth and Van der Corput to the set of primes. Namely, Green \cite{G} showed that every $A\subseteq\mathbf{P}$ with positive relative upper density, i.e. $\limsup_{N\to\8}\frac{|A\cap[1, N]|}{|\mathbf{P}\cap[1, N]|}>0$ contains a non--trivial three--term arithmetic progression. At almost the same time Green and Tao \cite{GT} proved the counterpart of Szemer\'{e}di's theorem \cite{Ser1} in the primes. More precisely, they established the existence of arbitrarily long arithmetic progressions in subsets of the primes having positive relative upper density. It is worth  pointing out that Green's theorem \cite{G} provides some quantitative result. Namely, it shows that if $|A\cap[1, N]|\ge CN(\log\log\log\log\log N)^{1/2}(\log N)^{-1}(\log\log\log\log N)^{-1/2}$ for some $N\ge N_0$, ($N_0\in\N$ and $C>0$ are absolute constants) then $A\cap[1, N]$ contains a non-trivial arithmetic progression of length three. The lower bound has been subsequently relaxed to $N\log\log\log N(\log N)^{-1}(\log\log N)^{-1/3}$ by Helfgott and De Roton \cite{HdR}, and recently to $N(\log N)^{-1}(\log\log N)^{-1+o(1)}$ by Naslund \cite{Nas}.

Finally, it should be emphasized  that there are also interesting  random constructions of sparse subsets of integers which contain  non--trivial three--term arithmetic progressions, see  \cite{KLR}, \cite{HL} and the references given there or recent paper of Conlon and Gowers \cite{CG}, which introduces new very powerful methods.

In spite of the fact that nowadays our knowledge of arithmetic structure of the set of prime numbers becomes satisfactory, not much has been developed for the set of Piatetski--Shapiro primes $\mathbf{P}_{\g}$ of fixed type $\g<1$ ($\g$ is sufficiently close to $1$), i.e.
$$\mathbf{P}_{\g}=\{p\in\mathbf{P}: \exists_{n\in\N}\ p=\lfloor n^{1/{\g}}\rfloor\}.$$
 In 1953 Piatetski--Shapiro \cite{PS} (see also \cite{GK}) established the asymptotic formula
 $$|\mathbf{P}_{\g}\cap[1, x]|\sim\frac{x^{\g}}{\log x}  \ \ \mbox{as \ $x\to\8$},$$ for every $\g\in(11/12, 1)$, which obviously implies that $\mathbf{P}_{\g}$ has a vanishing relative upper density in $\mathbf{P}$.
  It is worth emphasizing that the range $\g\in(11/12, 1)$ in the asymptotic formula of Piatetski--Shapiro \cite{PS} was improved by Kolesnik \cite{Kol}, Graham (unpublished), Leitmann (unpublished), Heath--Brown \cite{HB}, Kolesnik \cite{Kol1}, Liu--Rivat \cite{LR}, and recently  by Rivat and Sargos \cite{RS} for $\gamma\in(2426/2817, 1)$. This is the best known result to date.

  However, more to the point, it can be observed that neither Green \cite{G} nor Green and Tao \cite{GT} theorem does  settle if $\mathbf{P}_{\g}$ contains non--trivial arithmetic progressions of length at least three, since $\mathbf{P}_{\g}$ has zero density inside $\mathbf{P}$.

 Therefore, being motivated by this observation and the great recent achievements in the field of additive combinatorics, we are going to prove, in this paper, a counterpart of Roth's theorem for the Piatetski--Shapiro primes.
 \begin{thm}\label{PSthm}
 Assume that $\g\in(71/72, 1)$, then every $A\subseteq \mathbf{P}_{\g}$ with positive relative upper density, i.e. $\limsup_{N\to\8}\frac{|A\cap[1, N]|}{|\mathbf{P}_{\g}\cap[1, N]|}>0$ contains a non--trivial three--term arithmetic progression.
 \end{thm}
 However, the proof of Theorem \ref{PSthm} will follow from much more general Theorem \ref{Roththm} where we are going to study subsets of the prime numbers of the form $$\mathbf{P}_{h}=\{p\in\mathbf{P}:
\exists_{n\in\N}\ p=\lfloor h(n)\rfloor\},$$ where $h$ is an appropriate
function. Before we formulate Theorem \ref{Roththm} we need to introduce the definition of functions $h$, which we will consider. But throughout the paper, we encourage the reader to bear in mind the set of Piatetski--Shapiro primes as a principal example which will allow us to get a better understanding of further generalizations.

Throughout the whole paper, unless otherwise stated, we will use the convention that $C > 0$ stands for a large positive constant whose value may vary from occurrence to occurrence.
For two quantities $A>0$ and $B>0$ we say that $A\lesssim B$ ($A\gtrsim B$) if there exists an absolute constant $C>0$ such that $A\le CB$ ($A\ge CB$). We will write $A\lesssim_{\d} B$ ($A\gtrsim_{\d} B$) to indicate that
the constant $C>0$ depends on some $\d>0$. If $A\lesssim B$ and $A\gtrsim B$ hold simultaneously then we will shortly write that $A\simeq B$.
\begin{defn}\label{defn}
Let $c\in[1, 2)$ and $\mathcal{F}_c$ be the family of all functions $h:[x_0, \8)\mapsto [1, \8)$ (for some $x_0\ge1$)  satisfying
\begin{enumerate}
\item[(i)] $h\in \mathcal{C}^3([x_0, \8))$ and
$$h'(x)>0,\ \ \ \ h''(x)>0, \ \ \mbox{for every \  $x\ge x_0$.}$$
\item[(ii)] There exists a real valued function $\vartheta\in\mathcal{C}^2([x_0, \8))$ and a constant $C_h>0$ such that
\begin{align}\label{eq1}
  h(x)=C_hx^c\ell_h(x), \ \ \mbox{where}\ \ \ell_h(x)=e^{\int_{x_0}^x\frac{\vt(t)}{t}dt}, \ \ \mbox{for every \  $x\ge x_0$,}
\end{align}
and if $c>1$, then
\begin{align}\label{eq2}
  \lim_{x\to\8}\vartheta(x)=0,\ \ \lim_{x\to\8}x\vartheta'(x)=0,\ \ \lim_{x\to\8}x^2\vartheta''(x)=0.
\end{align}
  \item[(iii)] If $c=1$, then $\vt(x)$ is positive, decreasing and  for every $\varepsilon>0$
  \begin{align}\label{eq3}
    \frac{1}{\vt(x)}\lesssim_{\varepsilon}x^{\varepsilon}, \ \ \mbox{and} \ \ \lim_{x\to\8}\frac{x}{h(x)}=0.
  \end{align}
  Furthermore,
  \begin{align}\label{eq4}
  \lim_{x\to\8}\vartheta(x)=0,\ \ \lim_{x\to\8}\frac{x\vartheta'(x)}{\vt(x)}=0,\ \ \lim_{x\to\8}\frac{x^2\vartheta''(x)}{\vt(x)}=0.
\end{align}
\end{enumerate}
\end{defn}

From now on, having defined the family $\mathcal{F}_c$, we will focus our attention  on  subsets of the prime numbers $\mathbf{P}$ which have the following form
$$\{p\in\mathbf{P}: \exists_{n\in\N}\ p=\lfloor h(n)\rfloor\},$$
where $h\in\mathcal{F}_c$.
Let $\vp:[h(x_0), \8)\mapsto[1, \8)$ be the inverse function to $h$ and $\pi_h(x)$ denotes the cardinality of the set $\mathbf{P}_{h, x}=\mathbf{P}_{h}\cap[1, x]$. The family $\mathcal{F}_c$ was introduced by Leitmann in \cite{Leit} where he showed
$$\pi_h(x)\sim\frac{\vp(x)}{\log x}  \ \ \mbox{as \ $x\to\8$},$$
for every  $h\in\mathcal{F}_c$ with $c\in[1, 12/11)$. However, it is worth mentioning that originally Leitmann's definition of his family was more complicated. At the expense of  additional effort we have eliminated these complications keeping the same class of functions and having more handy formulations.

Among the functions belonging to the family $\mathcal{F}_c$ are (up to multiplicative constant $C_h>0$)
\begin{align*}
  h_1(x)=x^c\log^Ax,\ \ h_2(x)=x^ce^{A\log^Bx},\ \ h_2(x)=x\log^Cx, \ \ h_4(x)=xe^{C\log^Bx}, \ \ h_5(x)=xl_m(x),
\end{align*}
where $c\in(1, 2)$, $A\in\R$, $B\in(0, 1)$, $C>0$, $l_1(x)=\log x$ and $l_{m+1}(x)=\log(l_m(x))$, for $m\in\N$.

Our main result is the following.
 \begin{thm}\label{Roththm}
 Assume that $c\in[1, 72/71)$, $h\in\mathcal{F}_c$. Then every $A\subseteq \mathbf{P}_{h}$ with positive relative upper density, i.e. $\limsup_{N\to\8}\frac{|A\cap[1, N]|}{|\mathbf{P}_{h}\cap[1, N]|}>0$ contains a non--trivial three--term arithmetic progression.
 \end{thm}
Taking $h(x)=x^{1/\g}$ and $\g\in(71/72, 1)$ in the above theorem we immediately obtain Theorem \ref{PSthm}. The proof of Theorem \ref{Roththm} is based to a large extent on the ideas of Green pioneered in \cite{G}, see also \cite{GT1}. The main ingredient will be a variant so--called Hardy--Littlewood majorant property for the set $\mathbf{P}_{h}$. Namely,
 \begin{thm}\label{HLthm}
Assume that $c\in[1, 16/15)$, $\g=1/c$, $h\in\mathcal{F}_c$. Suppose that $(a_n)_{n\in\N}$ is a sequence of complex numbers such that
$|a_n|\le1$ for any $n\in\N$. Then for any $r>\frac{26-24\g}{16\g-15}$ we have
\begin{align}\label{HL}
  \bigg\|\sum_{p\in\mathbf{P}_{h, N}}a_pe^{2\pi i p\xi}\bigg\|_{L^r(\T, d\xi)}\lesssim_{r, \g} \bigg\|\sum_{p\in\mathbf{P}_{h, N}}e^{2\pi i p\xi}\bigg\|_{L^r(\T, d\xi)},
\end{align}
where the implied constant depends on $r$ and on $\g$, but does not depend on $N\in\N$.
\end{thm}
In fact, in order to get Theorem \ref{HLthm}, we prove likewise in \cite{G}, a somewhat stronger result (see Theorem \ref{bgthm1}), which we call a restriction theorem for the set $\mathbf{P}_h$. The strategy of our proof (Theorem \ref{HLthm} or Theorem \ref{bgthm1}) is extremely simple. We shall reduce the estimate over $p\in\mathbf{P}_{h, N}$ in Theorem \ref{HLthm} to the estimate over $p\in\mathbf{P}_{N}=\mathbf{P}\cap[1, N]$ and use the result of Green \cite{G}. Our task then, will be reduced to study the error term. For this purpose we have to prove the following.
\begin{lem}\label{formlem}
Assume that $c\in[1, 16/15)$, $h\in\mathcal{F}_c$, $\vp$ be its inverse and $\g=1/c$. Let $q\in\N$ and $0\le a\le q-1$ such that $(a, q)=1$. If $\chi>0$ satisfy $16(1-\g)+28\chi<1$, then there exists $\chi'>0$ such that for every $N\in\N$ and for every $\xi\in[0, 1]$
\begin{align}\label{form}
   \sum_{\genfrac{}{}{0pt}{}{p\in\mathbf{P}_{h, N}}{p\equiv a(\mathrm{mod}q)}}\vp'(p)^{-1}\log p\ e^{2\pi i \xi p}=\sum_{\genfrac{}{}{0pt}{}{p\in\mathbf{P}_{N}}{p\equiv a(\mathrm{mod}q)}} \log p\ e^{2\pi i \xi p}+O\big(N^{1-\chi-\chi'}\big).
\end{align}
The implied constant is independent of $\xi$ and $N\in\N$.
\end{lem}
Loosely speaking, the second sum in \eqref{form} represents the term which will be covered by the result of Green \cite{G}. The error term provides a decay which determines the range of  $r>\frac{26-24\g}{16\g-15}$ in Theorem \ref{HLthm}. In the proof of Lemma \ref{formlem} we will not use the circle method of Hardy and Littlewood, which was one of the main tools in Green's work. This is caused by the completely different nature of our problem. Our problem requires Van der Corput methods/inqualities to estimate trigonometric polynomials, instead of Weyl--Vinogradov's inequality. This is forced by the non--polynomial character of functions belonging to the family $\mathcal{F}_c$. A variant of formula \eqref{form}  was proved by Balog and Friedlander \cite{BF} and by Kumchev \cite{Kum} in the context of Piatetski--Shapiro primes. They used this result to show that the ternary Goldbach problem has a solution in the Piatetski--Shapiro primes (with different parameters $\g$) instead of primes. Their theorem  has been recently extended by the author \cite{M} to the functions belonging to $\mathcal{F}_c$. On the other hand using some variant of \eqref{form} we were able to establish in \cite{M} $L^r$ -- pointwise ergodic theorems along the set $\mathbf{P}_h$ for any $r>1$. The proof of Lemma \ref{formlem} will be a co--product of methods developed by Heath--Brown \cite{HB} with the techniques from the standard proof of Vinogradov's inequality from the ternary Goldbach problem, see \cite{GK} or \cite{Nat}. However, our approach differs from the one presented by Balog and Friedlander, or Kumchev due to the complexity of functions $h\in\mathcal{F}_c$. We obtain a qualitative improvement of their result at the expense of loss of quantitative nature of their lemma. We encourage the reader to compare Lemma \ref{formlem} with the results from \cite{BF} and \cite{Kum}.


The paper is organized as follows. In Section \ref{sectf} we give the necessary properties of function $h\in\mathcal{F}_c$ and its inverse $\vp$. In Section \ref{secttool} we gathered all the tools which will be used in the other sections. Assuming momentarily Lemma \ref{formlem} we give  proofs of Theorem \ref{HLthm} and Theorem \ref{Roththm} in Section \ref{sectres} and Section \ref{sectroth} respectively. In the penultimate section we estimate some exponential sums which allows us to give   proof of Lemma \ref{formlem}, which has been postponed to Section \ref{sectformlem}.

\section*{Acknowledgements}
I would like to thank Christoph Thiele for drawing to my attention the article of Ben Green \cite{G}, which turned out to be invaluable for this paper.

\section{Basic properties of functions $h$ and $\vp$}\label{sectf}
In this section we formulate all necessary properties of function $h\in\mathcal{F}_c$ and its inverse $\vp$.
We begin with the following.
\begin{lem}\label{filem}
Assume that $c\in[1, 2)$ and $h\in\mathcal{F}_c$. Then for every $i=1, 2, 3$ there exists a function $\vt_i:[x_0, \8)\mapsto\R$ such that
 \begin{align}\label{heq}
   xh^{(i)}(x)=h^{(i-1)}(x)(\a_i+\vt_i(x)),\ \ \mbox{for every \ $x\ge x_0$,}
 \end{align}
 where $\a_i=c-i+1$, $\vt_1(x)=\vt(x)$,
 \begin{align}\label{heq1}
   \vt_i(x)=\vt_{i-1}(x)+\frac{x\vt_{i-1}'(x)}{\a_{i-1}+\vt_{i-1}(x)},\ \ \mbox{for $i=2, 3$}\ \ \mbox{and}\ \ \lim_{x\to\8}\vartheta_i(x)=0,\ \ \mbox{for $i=1, 2, 3$}.
 \end{align}
 If $c=1$, then there exist constants $0<c_1\le c_2$ and a function $\vr:[x_0, \8)\mapsto[c_1, c_2]$, such that
  \begin{align}\label{vt2c1}
   \vt_2(x)=\vt(x)\vr(x),\ \ \mbox{for every \ $x\ge x_0$ \ and } \lim_{x\to\8}\frac{x\vt_2'(x)}{\vt_2(x)}=0.
 \end{align}
 In particular \eqref{heq} with $i=2$ reduces to
 \begin{align}\label{heqc1}
   xh''(x)=h'(x)\vt(x)\vr(x),\ \ \mbox{for every \ $x\ge x_0$.}
 \end{align}
 The cases for $i=1, 3$ remain unchanged.
\end{lem}
\begin{proof}
We may assume, without loss of generality that the constant $C_h=1$. Since $h(x)=x^c\ell_h(x)$ and $x\ell'_h(x)=\ell_h(x)\vt(x)$, then
\begin{align*}
  h'(x)=x^{c-1}\ell_h(x)(c+\vt(x)),
\end{align*}
thus taking $\vt_1(x)=\vt(x)$ we obtain \eqref{heq} for $i=1$. Generally, we see that if \eqref{heq} holds for $i-1\ge1$ instead of $i$, then this guarantees that $\frac{h^{(i-2)}(x)}{x}=\frac{h^{(i-1)}(x)}{\a_{i-1}+\vt_{i-1}(x)}$ holds for all $x\ge x_0$, and we have
\begin{align*}
 h^{(i)}(x)&=\left(\frac{h^{(i-1)}(x)}{x}-\frac{h^{(i-2)}(x)}{x^2}\right)(\a_{i-1}
 +\vt_{i-1}(x))+\frac{h^{(i-2)}(x)}{x}\vt_{i-1}'(x)\\
 &=\frac{h^{(i-1)}(x)}{x}\left(\left(1-\frac{1}{\a_{i-1}+\vt_{i-1}(x)}\right)(\a_{i-1}
 +\vt_{i-1}(x))+\frac{x\vt_{i-1}'(x)}{\a_{i-1}+\vt_{i-1}(x)}\right)\\
 &=\frac{h^{(i-1)}(x)}{x}\left(c-i+1+\vt_{i-1}(x)+\frac{x\vt_{i-1}'(x)}{\a_{i-1}+\vt_{i-1}(x)}\right).
\end{align*}
Thus we have proved that \eqref{heq} holds with $\a_i=c-i+1$ and $\vt_i(x)=\vt_{i-1}(x)+\frac{x\vt_{i-1}'(x)}{\a_{i-1}+\vt_{i-1}(x)}$. We now easily see that
\begin{align*}
 \vt_{i}'(x)=\vt'_{i-1}(x)+\frac{(\vt'_{i-1}(x)+x\vt''_{i-1}(x))(\a_{i-1}+\vt_{i-1}(x))
 -x\vt'_{i-1}(x)^2}{(\a_{i-1}+\vt_{i-1}(x))^2},
\end{align*}
and consequently $\lim_{x\to\8}\vt_i(x)=0$ for any $i=1, 2, 3$ by \eqref{eq2}.

In order to get \eqref{vt2c1} and \eqref{heqc1} we note that
\begin{align*}
\vt_2(x)=\vt(x)\left(1+\frac{x\vt'(x)}{\vt(x)(1+\vt(x))}\right).
\end{align*}
Taking $\vr(x)=1+\frac{x\vt'(x)}{\vt(x)(1+\vt(x))}$ we immediately see that there exist constants
$0<c_1\le c_2$ such that $c_1\le\vr(x)\le c_2$, by \eqref{eq4}.
The calculations stated above yield $xh'''(x)=h''(x)(-1+\vt_3(x))$ where $\vt_3(x)=\vt(x)
  +\frac{x\vt'(x)}{1+\vt(x)}+\frac{x\vt_2'(x)}{\vt_2(x)}$. The only point remaining concerns the behaviour
  of $\vt_3(x)$. We only need to prove that $\lim_{x\to\8}\frac{x\vt_2'(x)}{\vt_2(x)}=0$. Namely, by \eqref{eq4} we have
  \begin{align*}
    \lim_{x\to\8}\frac{x\vt_2'(x)}{\vt_2(x)}=\lim_{x\to\8}\frac{\frac{x\vt'(x)(1+\vt(x))}{\vt(x)}
    +\frac{(x\vt'(x)+x^2\vt''(x))(1+\vt(x))-x^2\vt'(x)^2}{\vt(x)(1+\vt(x))}}
    {1+\vt(x)+\frac{x\vt'(x)}{\vt(x)}}=0.
  \end{align*}
The  proof of the lemma is completed.
\end{proof}

\begin{lem}\label{formfunlem}
Assume that $c\in[1, 2)$, $h\in\mathcal{F}_c$, $\g=1/c$ and let $\vp:[h(x_0), \8)\mapsto[x_0, \8)$ be its inverse. Then there exists a function $\te:[h(x_0),\8)\mapsto\R$ such that $x\vp'(x)=\vp(x)(\g+\te(x))$ and
  \begin{align}\label{funfi}
  \vp(x)=x^{\g}\ell_{\vp}(x),\ \ \ \mbox{where}\ \ \ \ell_{\vp}(x)=e^{\int_{h(x_0)}^x\frac{\te(t)}{t}dt+D},
\end{align}
  for every $x\ge h(x_0)$, where $D=\log(x_0/h(x_0)^{\g})$ and $\lim_{x\to\8}\te(x)=0$. Moreover,
  \begin{align}\label{tetadef}
  \te(x)=\frac{1}{(c+\vartheta(\vp(x)))}-\g
  =-\frac{\vartheta(\vp(x))}{c(c+\vartheta(\vp(x)))}.
\end{align}
  Additionally, for every $\e>0$
  \begin{align}\label{slowhfi}
    \lim_{x\to\8}x^{-\e}L(x)=0,\ \ \ \mbox{and}\ \ \ \lim_{x\to\8}x^{\e}L(x)=\8,
  \end{align}
  where $L(x)=\ell_h(x)$ or $L(x)=\ell_{\vp}(x)$. In particular, for every $\e>0$
  \begin{align}\label{ratefi}
    x^{\g-\e}\lesssim_{\e}\vp(x),\ \ \ \mbox{and}\ \ \ \lim_{x\to\8}\frac{\vp(x)}{x}=0.
  \end{align}
  Finally, $x\mapsto x\vp(x)^{-\d}$ is increasing for every $\d< c$, (if $c=1$, even $\d\le1$ is allowed) and for every $x\ge h(x_0)$ we have
  \begin{align}\label{compfi}
    \vp(x)\simeq\vp(2x),\ \ \mbox{and}\ \ \vp'(x)\simeq\vp'(2x).
  \end{align}
\end{lem}
\begin{proof}
Lemma \ref{filem} yields that $\lim_{x\to\8}\frac{xh'(x)}{h(x)}=c$, thus taking $\te(x)=\frac{x\vp'(x)}{\vp(x)}-\g$ we see that $\lim_{x\to\8}\te(x)=0$ and $x\vp'(x)=\vp(x)(\g+\te(x))$.
Now observe that
\begin{align*}
  \frac{\vp'(x)}{\vp(x)}=\frac{\g}{x}+\frac{\te(x)}{x}.
\end{align*}
Thus \eqref{funfi} with $D=\log(x_0/h(x_0)^{\g})$ follows from
\begin{align*}
  \log \vp(x)=\int_{h(x_0)}^x\frac{\vp'(t)}{\vp(t)}dt+\log x_0=\log x^{\g}+\int_{h(x_0)}^x\frac{\te(t)}{t}dt
  +\log x_0-\log h(x_0)^{\g}.
\end{align*}
In view of $\vp(x)h'(\vp(x))=h(\vp(x))(c+\vartheta(\vp(x)))=x(c+\vartheta(\vp(x)))$ we easily get \eqref{tetadef} since
\begin{align*}
  \te(x)=\frac{x\vp'(x)}{\vp(x)}-\g=\frac{x}{\vp(x)h'(\vp(x))}-\g=\frac{1}{(c+\vartheta(\vp(x)))}-\g
  =-\frac{\vartheta(\vp(x))}{c(c+\vartheta(\vp(x)))}.
\end{align*}

To prove \eqref{slowhfi} we may assume, without loss of generality, that $|\vartheta(x)|\le\e/2$ for every $x\ge x_0$, and observe
\begin{align*}
  x^{-\e}e^{\int_{x_0}^x\frac{\vartheta(t)}{t}dt+C}\le x^{-\e}e^{\frac{\e}{2}\int_{x_0}^x\frac{dt}{t} +C}=x^{-\e}x^{\e/2}e^C\ _{\overrightarrow{x\to\8}}\ 0.
\end{align*}
On the other hand
\begin{align*}
  x^{\e}e^{\int_{x_0}^x\frac{\vartheta(t)}{t}dt+C}\ge x^{\e}e^{-\frac{\e}{2}\int_{x_0}^x\frac{dt}{t} +C}=x^{\e}x^{-\e/2}e^C\ _{\overrightarrow{x\to\8}}\ \8.
\end{align*}
The rest of the proof (the case of $\ell_{\vp}$) runs as before. The first inequality in \eqref{ratefi} can be drawn from \eqref{slowhfi}, whereas the limit in \eqref{ratefi} is equal to $0$ by \eqref{eq3}, since
$\lim_{x\to\8}\frac{\vp(x)}{x}=\lim_{x\to\8}\frac{\vp(x)}{h(\vp(x))}=0$. Now we show that $x\mapsto x\vp(x)^{-\d}$ is increasing for every $\d< c$. Indeed,
$$\left(\frac{x}{\vp(x)^{\d}}\right)'=\frac{\vp(x)^{\d}-\d x\vp(x)^{\d-1}\vp'(x)}{\vp(x)^{2\d}}=
\frac{1-\d\g-\d\te(x)}{\vp(x)^{\d}}>0\ \Longleftrightarrow\ \d< c.$$
If $c=1$ then $\d\le1$ is allowed, since  $\te(x)<0$ by \ref{tetadef}.
The proof will be finished if we show \eqref{compfi}. It suffices to show \eqref{compfi} only for large $x\ge h(x_0)$, therefore we may assume that $|\te(x)|\le\g/4$ and $|\te(2x)|\le\g/4$ and observe
\begin{align*}
  \vp(x)\le\vp(2x)=\frac{2x\vp'(2x)}{\g+\te(2x)-\te(x)/2+\te(x)/2}\le\frac{2x\vp'(x)}{\g/2+\te(x)/2}\lesssim
  \vp(x).
\end{align*}
The proof of Lemma \ref{formfunlem} is completed.
\end{proof}
The next lemma provides a very useful formula expressing the characteristic function of the set
$\mathbf{P}_{h}$ in a more handy form.
\begin{lem}\label{intlem}
Assume that $h\in\mathcal{F}_c$ and let $\vp:[h(x_0), \8)\mapsto[x_0, \8)$ be its inverse. Then
\begin{align}\label{intlemform}
  p\in\mathbf{P}_{h} \Longleftrightarrow\ \lfloor-\vp(p)\rfloor-\lfloor-\vp(p+1)\rfloor=1,
\end{align}
for all sufficiently large $p\in\mathbf{P}_{h}$.
\end{lem}
\begin{proof}
First of all notice that  $h'(x)\ge1$ for every large enough $x\ge x_0$, thus
$h(x+1)-h(x)\ge 1.$
It suffices to show that
$$\exists_{n\in\N}\ p=\lfloor h(n)\rfloor \Longleftrightarrow\ \lfloor-\vp(p)\rfloor-\lfloor-\vp(p+1)\rfloor=1.$$
Assume that $p=\lfloor h(n)\rfloor$, this is equivalent to
$p\le h(n)<p+1\Longleftrightarrow\vp(p)\le n<\vp(p+1)$,
and implies that $\vp(p+1)\le \vp(h(n)+1)\le \vp(h(n+1))=n+1$, hence
$-n-1\le-\vp(p+1)<-n\le-\vp(p),$
 and we get $\lfloor-\vp(p+1)\rfloor=-n-1$ and $-n\le\lfloor-\vp(p)\rfloor $. Thus we see
\begin{align*}
  1=n+1-n&\le\lfloor-\vp(p)\rfloor-\lfloor-\vp(p+1)\rfloor
  <\vp(p+1)-\vp(p)+1=\int_p^{p+1}\vp'(x)dx+1<2,
\end{align*}
for all sufficiently large $p\in\mathbf{P}_{h}$, since $\vp'(x)=\frac{1}{h'(\vp(x))}$ and $\lfloor-\vp(p+1)\rfloor>-\vp(p+1)-1$.

Now assume that $\lfloor-\vp(p)\rfloor-\lfloor-\vp(p+1)\rfloor=1,$ hence
$\lfloor-\vp(p)\rfloor=1+\lfloor-\vp(p+1)\rfloor\le -\vp(p),$ thus
$$\vp(p)\le -\lfloor-\vp(p+1)\rfloor-1< \vp(p+1)+1-1=\vp(p+1).$$
Therefore, taking $n=-\lfloor-\vp(p+1)\rfloor-1$ we obtain
$$\vp(p)\le n<\vp(p+1)\Longleftrightarrow p\le h(n)< p+1,$$
as desired. The proof of Lemma \ref{intlem} is completed.
\end{proof}

We will look more closely at the function $\vp$ being the inverse function to the function $h\in\mathcal{F}_c$ and we collect all required properties its derivatives in the following.
\begin{lem}\label{funlemfi}
Assume that $c\in[1, 2)$, $h\in\mathcal{F}_c$, $\g=1/c$ and let $\vp:[h(x_0), \8)\mapsto[x_0, \8)$ be its inverse. Then
 for every $i=1, 2, 3,$ there exists  a function $\theta_i:[h(x_0), \8)\mapsto\R$ such that
 \begin{align}\label{fiequat}
   x\vp^{(i)}(x)=\vp^{(i-1)}(x)(\b_i+\theta_i(x)), \ \ \mbox{for every \  $x\ge h(x_0)$,}
 \end{align}
  where $\b_i=\g-i+1$ and $\lim_{x\to\8}\te_i(x)=0.$
If $c=1$, then there exists a positive function $\s:[h(x_0), \8)\mapsto(0, \8)$ and a function $\t:[h(x_0), \8)\mapsto \R$ such that \eqref{fiequat} with $i=2$ reduces to
\begin{align}\label{fiequat1}
  x\vp''(x)=\vp'(x)\s(x)\t(x),\ \ \mbox{for every \  $x\ge h(x_0)$ \ and } \lim_{x\to\8}\frac{x\te_2'(x)}{\te_2(x)}=0.
\end{align}
The cases for $i=1, 3$ remain unchanged. Moreover, $\s(x)$  is decreasing, $\lim_{x\to\8}\s(x)=0,$ $\s(2x)\simeq\s(x),$ and  $\s(x)^{-1}\lesssim_{\varepsilon}x^{\varepsilon},$
for every $\varepsilon>0$. Finally, there are constants $0<c_3\le c_4$ such that  $c_3\le-\t(x)\le c_4$ for every $x\ge h(x_0)$.
\end{lem}

\begin{proof}
The proof is based on simple computations. However, for the convenience of the reader we have decided to give the details. In fact, \eqref{fiequat} for $i=1$ with $\te_1(x)=\te(x)$, has been shown in  Lemma \ref{formfunlem}. Arguing likewise in the proof of Lemma \ref{filem} we obtain \eqref{fiequat} for $i=2, 3$. More precisely,
  \begin{align}\label{te1}
    \te_1(x)=\te(x)=-\frac{\vartheta(\vp(x))}{c(c+\vartheta(\vp(x)))}=\frac{1}{c+\vartheta(\vp(x))}-\g,
  \end{align}
  \begin{align}\label{te2}
  \te_2(x)&=\te(x)+\frac{x\te'(x)}{\g+\te(x)}=\frac{1}{c+\vartheta(\vp(x))}-\g-
  \frac{\vartheta'(\vp(x))\vp(x)}{(c+\vartheta(\vp(x)))^2},
\end{align}
since $\te'(x)=\left(\frac{1}{c+\vartheta(\vp(x))}-\g\right)'
  =-\frac{\vartheta'(\vp(x))\vp'(x)}{(c+\vartheta(\vp(x)))^2},$ and
\begin{align}\label{te3}
  \te_3(x)=\te(x)+\frac{x\te'(x)}{\g+\te(x)}+\frac{x\te_2'(x)}{\g-1+\te_2(x)},
\end{align}
where
\begin{align}\label{te2p}
  \te_2'(x)=-\frac{(\vt''(\vp(x))\vp(x)+2\vt'(\vp(x)))(c+\vartheta(\vp(x)))
  -2\vt'(\vp(x))^2\vp(x)}{(c+\vartheta(\vp(x)))^3}\vp'(x),
\end{align}
since
\begin{align*}
  \te_2'(x)
  &=\left(\frac{1}{c+\vartheta(\vp(x))}-\g-
  \frac{\vartheta'(\vp(x))\vp(x)}{(c+\vartheta(\vp(x)))^2}\right)'
  =-\frac{\vt'(\vp(x))\vp'(x)}{(c+\vartheta(\vp(x)))^2}\\
  &-\frac{(\vt''(\vp(x))\vp'(x)\vp(x)+\vt'(\vp(x))\vp'(x))(c+\vt(\vp(x)))
  -2\vt'(\vp(x))^2\vp(x)\vp'(x)}{(c+\vt(\vp(x)))^3}.
\end{align*}
The proof will be completed, if we elaborate the case $c=1$. We know that $x\vp''(x)=\vp'(x)\te_2(x)$,
with
\begin{align*}
  \te_2(x)&=-\frac{\vartheta(\vp(x))}{1+\vartheta(\vp(x))}-
  \frac{\vartheta'(\vp(x))\vp(x)}{(1+\vartheta(\vp(x)))^2}
  =\vartheta(\vp(x))\left(-\frac{1}{1+\vartheta(\vp(x))}-
  \frac{\vartheta'(\vp(x))\vp(x)}{\vartheta(\vp(x))(1+\vartheta(\vp(x)))^2}\right).
\end{align*}
Therefore \eqref{fiequat1} is proved with $\s(x)=\vt(\vp(x))$ and
$\t(x)=-\left(\frac{1}{1+\vartheta(\vp(x))}+
\frac{\vartheta'(\vp(x))\vp(x)}{\vartheta(\vp(x))(1+\vartheta(\vp(x)))^2}\right).$
In order to show that $\s(2x)\simeq\s(x)$ it is enough to prove that $\vt(2x)\simeq\vt(x)$.
Notice that for some $\xi_x\in(0, 1)$ we have
\begin{align*}
  \left|\frac{\vt(2x)}{\vt(x)}-1\right|=\left|\frac{(x+\xi_xx)\vt'(x+\xi_xx)}{\vt(x+\xi_xx)}\right|
  \frac{x}{x+\xi_xx}\frac{\vt(x+\xi_xx)}{\vt(x)}\le\left|\frac{(x+\xi_xx)\vt'(x+\xi_xx)}{\vt(x+\xi_xx)}\right|
  \ _{\overrightarrow{x\to\8}}\ 0,
\end{align*}
since $\vt(x)$ is decreasing.
It is easy to see that $$\s(x)^{-1}\lesssim x^{\e}, \ \ \mbox{for every $\e>0$,}$$
since $\vt(x)^{-1}\lesssim_{\varepsilon}x^{\varepsilon}$ for every $\varepsilon>0$ and by \eqref{ratefi}. Furthermore, there exist $0<c_3\le c_4$ such that $c_3\le -\t(x)\le c_4$ for every $x\ge h(x_0)$, by \eqref{eq4}. The only what is left is to verify that $\lim_{x\to\8}\frac{x\te_2'(x)}{\te_2(x)}=0$. Indeed, by \eqref{eq4} we have
\begin{align*}
  \lim_{x\to\8}\frac{x\te_2'(x)}{\te_2(x)}=\lim_{x\to\8}\frac{\frac{(\vt''(\vp(x))\vp(x)^2+2\vt'(\vp(x))\vp(x))
  (1+\vartheta(\vp(x)))
  -2\vt'(\vp(x))^2\vp(x)^2}{\vartheta(\vp(x))(1+\vartheta(\vp(x)))^4}}{\frac{1}{1+\vartheta(\vp(x))}+
  \frac{\vartheta'(\vp(x))\vp(x)}{\vartheta(\vp(x))(1+\vartheta(\vp(x)))^2}}=0.
\end{align*}
This completes the proof.
\end{proof}

\section{Necessary tools}\label{secttool}

Here we state all lemmas and fact from analytic number theory which will be used in the sequel.
All of these results can be found in \cite{GK}, \cite{IK} and \cite{Nat}.
\subsection{Van der Corput's results}
\begin{lem}[Van der Corput]\label{vdc}
Assume that $a, b\in\R$ and $a<b$. Let $F\in\mathcal{C}^2([a, b])$ be a real valued function and let $I$ be a subinterval of $[a, b]$. If there exists $\eta>0$ and $r\ge 1$  such that
\begin{align*}
  \eta\lesssim |F''(x)|\lesssim r\eta,\ \ \mbox{for every \ $x\in I$,}
\end{align*}
then
$$\bigg|\sum_{k\in I}e^{2\pi i F(k)}\bigg|\lesssim r|I|\eta^{1/2}+\eta^{-1/2}.$$
\end{lem}
Proof of Lemma \ref{vdc} can be found in \cite{IK}, see Corollary 8.13, page 208.
\begin{lem}[Weyl \& Van der Corput inequality]\label{vdc1}
Let $H\ge1$ be fixed and $z_h\in\C$ be any complex number with $H<h\le 2H$ and $I\subseteq(H, 2H]$ be an interval. Then for every $R\in\N$ we have
$$\bigg|\sum_{h\in I}z_h\bigg|^2\le \frac{H+R}{R}\sum_{|r|\le R}\left(1-\frac{|r|}{R}\right)\sum_{h, h+r\in I}z_h\overline{z}_{h+r}.$$
\end{lem}
Proof of Lemma \ref{vdc1} can be found in \cite{HB} Lemma 5, page 258.
\subsection{Fourier expansions}
Let us define $\Phi(x)=\{x\}-1/2$ and expand $\Phi$ in the Fourier series (see \cite{HB} Section 2), i.e. we obtain
\begin{align}\label{four1}
  \Phi(t)=\sum_{0<|m|\le M}\frac{1}{2\pi i m}e^{-2\pi imt}+O\left(\min\left\{1, \frac{1}{M\|t\|}\right\}\right),
\end{align}
for $M>0$, where $\|t\|=\min_{n\in\Z}|t-n|$ is the distance of $t\in\R$ to the nearest integer. Parameter $M$ will give us some margin of flexibility in our further calculations and will allow us to produce the estimates with  the decay acceptable for us. Moreover,
\begin{align}\label{four2}
  \min\left\{1, \frac{1}{M\|t\|}\right\}=\sum_{m\in\Z}b_m e^{2\pi imt},
\end{align}
where
\begin{align}\label{fcoe2}
  |b_m|\lesssim \min\left\{\frac{\log M}{M}, \frac{1}{|m|}, \frac{M}{|m|^2}\right\}.
\end{align}
\subsection{Basic facts from analytic number theory}
Throughout the paper, we will use the following version  of summation by parts (see \cite{Nat} Theorem A.4, page 304.)
\begin{lem}\label{sbp}
Assume that $a$ and $b$ are real numbers such that $0\le a<b$. Let $u(n)$ and $g(n)$ be arithmetic functions and $U(t)=\sum_{a< n\le t}u(n)$ be the sum function of $u(n)$.  If $g\in\mathcal{C}^1([a, b])$, then
$$\sum_{a<n\le b}u(n)g(n)=U(b)g(b)-\int_a^bU(t)g'(t)dt.$$
\end{lem}
\noindent Let $\mu(n)$ be the M\"{o}bius function i.e.
$$\mu(n)=\left\{ \begin{array} {ll}
\ \ 1, & \mbox{if $n=1$,}\\
(-1)^k, & \mbox{if $n$ is the product of $k$ distinct primes,}\\
\ \ 0,& \mbox{if $n$ is divisible by the square of a prime.}
\end{array}
\right.$$
Therefore, $\mu(n)\not=0$ if and only if $n$ is square--free. Another important function for us will be
von Mangoldt's function $\Lambda(n)$ defined by
$$\Lambda(n)=\left\{ \begin{array} {ll}
\log p, & \mbox{if $n=p^m$ for some $m\in\N$ and $p\in\mathbf{P}$,}\\
\ \ 0,& \mbox{otherwise.}
\end{array}
\right.$$

For the estimates of exponential sums we will use
\begin{lem}[Vaughan's identity]\label{Vaug} Let $v, w$ be positive real numbers. If $v>n$ then
\begin{align}\label{vid}
  \Lambda(n)&=\sum_{\genfrac{}{}{0pt}{}{k_1k_2=n}{k_2\le w}}\log k_1\ \mu(k_2)-\sum_{\genfrac{}{}{0pt}{}{k_1k_2k_3=n}{k_2\le v, k_3\le w}}\Lambda(k_2)\mu(k_3)+\sum_{\genfrac{}{}{0pt}{}{k_1k_2=n}{k_1>v, k_2>w}}\Lambda(k_1)\bigg(\sum_{\genfrac{}{}{0pt}{}{d|k_2}{d>w}}\mu(d)\bigg)\\
  \nonumber&=\sum_{\genfrac{}{}{0pt}{}{kl=n}{l\le w}}\log k\ \mu(l)
  -\sum_{l\le vw}\sum_{kl=n}\Pi_{v, w}(l)+\sum_{\genfrac{}{}{0pt}{}{kl=n}{k>v, l>w}}\Lambda(k)\Xi_w(l),
\end{align}
where
\begin{align}\label{pixi}
  \Pi_{v, w}(l)=\sum_{\genfrac{}{}{0pt}{}{rs=l}{r\le v, s\le w}}\Lambda(r)\mu(s), \ \ \ \mbox{and}\ \ \ \ \Xi_{w}(l)=\sum_{\genfrac{}{}{0pt}{}{d|l}{d>w}}\mu(d).
\end{align}
\end{lem}
If $v=w$ (this will be our case) we will shortly write $\Pi_{v}(l)$ instead of $\Pi_{v, v}(l)$. Vaughan's identity will be critical for us. The proof of Lemma \ref{Vaug} can be found in \cite{IK} see Proposition 13.4, page 345 or in \cite{GK} Lemma 4.12, page 49.
\begin{thm}[Siegel--Walfisz]\label{swthm}
If $B>0$, $1\le q\le \log^BN$ and $(a, q)=1$, then
\begin{align}\label{sw}
  \psi(N; q, a)=\sum_{\genfrac{}{}{0pt}{}{p\in\mathbf{P}_{ N}}{p\equiv a(\mathrm{mod}q)}}\log p=\frac{N}{\phi(q)}+O\left(\frac{N}{\log^B N}\right),
\end{align}
for all $N\ge 2$, where $\phi$ denotes the Euler's function and the implied constant depends only on $B$.
\end{thm}
For the proof of Siegel--Walfisz Theorem we refer to \cite{IK}, Corollary 5.29, page 124. Now using Theorem \ref{swthm} and formula \eqref{form} we derive the following.
\begin{thm}\label{swlthm}
Assume that $c\in [1, 12/11)$, $\g=1/c$, $h\in\mathcal{F}_c$ and $\vp$ be its inverse.
If $B>0$, $1\le q\le \log^BN$ and $(a, q)=1$, then
\begin{align}
 \label{swl} \psi_h(N; q, a)&=\sum_{\genfrac{}{}{0pt}{}{p\in\mathbf{P}_{h, N}}{p\equiv a(\mathrm{mod}q)}}\log p=\frac{\vp(N)}{\phi(q)}+O\left(\frac{\vp(N)}{\log^B N}\right),\\
 \label{swl1} \pi_h(N; q, a)&=\sum_{\genfrac{}{}{0pt}{}{p\in\mathbf{P}_{h, N}}{p\equiv a(\mathrm{mod}q)}}1=\frac{1}{\phi(q)}\frac{\vp(N)}{\log N}+O\left(\frac{\vp(N)}{\log^2N}\right),
\end{align}
for all $N\ge 2$, where the implied constant depends only on $h$ and $B$.
\end{thm}
Theorem \ref{swlthm} was proved by Leitmann in \cite{Leit}. For $c\in[1, 16/15)$ the proof can be easily derived with the aid of formula \eqref{form} with $\xi=0$, summation by parts and \eqref{sw}.

 \section{Restriction theorem for the set $\mathbf{P}_h$ and the proof of theorem \ref{HLthm}}\label{sectres}
 This section is intended to prove Theorem \ref{bgthm1}, which we will call a restriction theorem for the set $\mathbf{P}_h$. The case of the prime numbers $\mathbf{P}$, see Theorem \ref{bgthm} below, was proved by Bourgain in \cite{B} and recently it has been rediscovered by Green \cite{G} in the context of arithmetic progressions. Throughout this section we will assume that $c\in[1, 16/15)$, $\g=1/c$, $h\in\mathcal{F}_c$ and $\vp$ is the inverse function to $h$. Moreover, $r'$ will denote the conjugate exponent to $r>1$, i.e. $\frac{1}{r}+\frac{1}{r'}=1$. We begin by recalling the results of Green from \cite{G} and by introducing necessary notation. Let $b\in\N\cup\{0\}$, $m, N\in\N$ such that $1\le m\le \log N$ and $0\le b\le m-1$ with $(b, m)=1$.  Define a set
 \begin{align*}
   \Lambda_{b, m, N}=\{0\le n\le N: mn+b\in\mathbf{P}\}.
 \end{align*}
  It is easy to see that $\Lambda_{b, m, N}$ has size about $mN/\phi(m)\log(mN)$ by Siegel--Walfisz theorem. Let us define a measure $\la_{b, m, N}$ on $\Lambda_{b, m, N}$ by setting
 \begin{align*}
   \la_{b, m, N}(n)=\left\{ \begin{array} {ll}
\frac{\phi(m)\log(mn+b)}{mN}, & \mbox{if $n\in\Lambda_{b, m, N}$,}\\
0, & \mbox{otherwise.}
\end{array}
\right.
 \end{align*}
 Let $\mathcal{F}_{\Z}[f](\xi)=\sum_{n\in\Z}f(n)e^{2\pi i \xi n}$ denotes the Fourier transform on $\Z$ and $\widehat{f}(n)=\int_{\mathbb{T}}f(\xi)e^{-2\pi i \xi n}d\xi$ denotes the Fourier transform on $\T$. For any measure space $X$ let $\mathcal{C}(X)$ denotes the space of all continuous functions on $X$ and define a linear operator $T: \mathcal{C}(\Lambda_{b, m, N})\to \mathcal{C}(\T)$ as follows
 \begin{align*}
   T(f)(\xi)=\mathcal{F}_{\Z}[f\la_{b, m, N}](\xi).
 \end{align*}
 \begin{thm}[Bourgain--Green]\label{bgthm}
 Suppose that $r>2$ is a real number. Then there is a finite constant $C_r>0$ such that for all functions
 $f\in L^2(\Lambda_{b, m, N}, \lambda_{b, m, N})$ we have
 \begin{align}\label{gthm1}
   \|Tf\|_{L^r(\T)}\le C_rN^{-1/r}\|f\|_{L^2(\Lambda_{b, m, N}, \lambda_{b, m, N})}.
 \end{align}
 \end{thm}
 Before we formulate a counterpart of Bourgain--Green's theorem for $\mathbf{P}_h$, let us introduce a set
 \begin{align*}
   \Lambda_{b, m, N}^h=\{0\le n\le N: mn+b\in\mathbf{P}_h\}.
 \end{align*}
 According to Theorem \ref{swlthm} the set $\Lambda_{b, m, N}^h$ has size comparable to $\vp(mN)/\phi(m)\log(mN)$. Therefore,  likewise above, it is natural to define a measure $\lambda_{b, m, N}^h$ on $\Lambda_{b, m, N}^h$ by setting
  \begin{align*}
   \la_{b, m, N}^h(n)=\left\{ \begin{array} {ll}
\frac{\phi(m)\log(mn+b)}{mN\vp'(mn+b)}, & \mbox{if $n\in\Lambda_{b, m, N}^h$,}\\
0, & \mbox{otherwise.}
\end{array}
\right.
 \end{align*}
 Our task now is to prove a restriction theorem for the set $\mathbf{P}_h$.
  \begin{thm}\label{bgthm1}
  Assume that $c\in[1, 16/15)$, $\g=1/c$, $h\in\mathcal{F}_c$ and $\vp$ be its inverse. Suppose that $r>\frac{26-24\g}{16\g-15}$ is a real number. Then there is a finite constant $C_{r, \g}>0$ such that for all functions
 $f\in L^2(\Lambda_{b, m, N}^h, \lambda_{b, m, N}^h)$ we have
 \begin{align}\label{gthm11}
   \|T_hf\|_{L^r(\T)}\le C_{r, \g}N^{-1/r}\|f\|_{L^2(\Lambda_{b, m, N}^h, \lambda_{b, m, N}^h)},
 \end{align}
 where $T_h: \mathcal{C}(\Lambda_{b, m, N}^h)\to \mathcal{C}(\T)$ is a linear operator given by
 \begin{align*}
   T_h(f)(\xi)=\mathcal{F}_{\Z}[f\la_{b, m, N}^h](\xi).
 \end{align*}
 \end{thm}
 \begin{proof}
 In the proof we will exploit Green's ideas from \cite{G} reducing the matters to Theorem \ref{bgthm}. As in \cite{G} the main tool will be $TT^*$ argument and an appropriate interpolation giving some restriction on the range of $r>\frac{26-24\g}{16\g-15}$. Let us briefly recall the role of $TT^*$ method. Firstly, notice that the relation
 \begin{align*}
   \langle T_hf, g\rangle_{L^2(\T)}=\int_{\T}\mathcal{F}_{\Z}[f\la_{b, m, N}^h](\xi)\overline{g(\xi)}d\xi=
   \sum_{n\in\Z}f(n)\overline{\widehat{g}(n)}\la_{b, m, N}^h(n)=\langle f, T^*_hg\rangle_{L^2(\Lambda_{b, m, N}^h, \lambda_{b, m, N}^h)},
 \end{align*}
 shows that the operator $T^*_h:\mathcal{C}(\T)^*\to\mathcal{C}(\Lambda_{b, m, N}^h)^*=\mathcal{C}(\Lambda_{b, m, N}^h)$ is given by
 \begin{align*}
   T^*_h(g)(n)=\widehat{g}(n)|_{\Lambda_{b, m, N}^h}=\widehat{g}(n)\cdot\mathbf{1}_{\Lambda_{b, m, N}^h}(n).
 \end{align*}
 Therefore, we have that the map $T_hT^*_h:\mathcal{C}(\T)^*\to\mathcal{C}(\T)^*$ is given by
 \begin{align*}
   T_hT^*_hf(\xi)=f*\mathcal{F}_{\Z}[\la_{b, m, N}^h](\xi).
 \end{align*}
 In the sequel we will consider the operator $T_hT^*_h$ as a mapping acting on $L^r(\T)$ spaces (it makes sense, since $L^r(\T)$ naturally embeds into $\mathcal{C}(\T)^*$ for any $r\ge1$).
 Now it is easy to see that
 \begin{align*}
   \|T_hf\|_{L^r(\T)}&
   \le \|T_hT_h^*\|^{1/2}_{L^{r'}(\T)\to L^{r}(\T)}\|f\|_{L^2(\Lambda_{b, m, N}^h, \lambda_{b, m, N}^h)} ,
 \end{align*}
 which is the heart of the matter and allows us to prove that  $T_hT_h^*$ satisfies the bound
 \begin{align*}
   \|T_hT_h^*\|_{L^{r'}(\T)\to L^{r}(\T)}\le C_{r, \g}N^{-2/r}.
 \end{align*}
 The strategy of our proof will be based on the reduction of our estimate to the estimate from Bourgain--Green's restriction theorem. For this purpose we will proceed as follows. For every $r>\frac{26-24\g}{16\g-15}\ge2$ observe that
 \begin{align*}
   \|T_hT_h^*f\|_{L^{r}(\T)}&=\|f*\mathcal{F}_{\Z}[\la_{b, m, N}^h]\|_{L^{r}(\T)}\\
   &\le
   \|f*\mathcal{F}_{\Z}[\la_{b, m, N}]\|_{L^{r}(\T)}
   +\|f*\mathcal{F}_{\Z}[\la_{b, m, N}^h-\la_{b, m, N}]\|_{L^{r}(\T)}\\
   &\le \|TT^*\|_{L^{r'}(\T)\to L^{r}(\T)}\|f\|_{L^{r'}(\T)}+\|f*\mathcal{F}_{\Z}[\la_{b, m, N}^h-\la_{b, m, N}]\|_{L^{r}(\T)}.
 \end{align*}
 In view of Bourgain--Green's theorem $\|TT^*\|_{L^{r'}(\T)\to L^{r}(\T)}\le C_rN^{-2/r}$ for every $r>2$. Therefore, it only remains to deal with the $L^r(\T)$ norm of the error term. Namely, we will be concerned with illustrating that for any $r>\frac{26-24\g}{16\g-15}$ we have
 \begin{align*}
   \|f*\mathcal{F}_{\Z}[\la_{b, m, N}^h-\la_{b, m, N}]\|_{L^{r}(\T)}\le C_{r, \g}N^{-2/r}\|f\|_{L^{r'}(\T)}.
 \end{align*}
 In order to achieve this bound it is convenient to find firstly, an $L^{2}(\T)\to L^{2}(\T)$ estimate, secondly an
 $L^{1}(\T)\to L^{\8}(\T)$ estimate and interpolate between them. Notice that
 \begin{align}\label{l2}
   \|f*\mathcal{F}_{\Z}[\la_{b, m, N}^h-\la_{b, m, N}]&\|_{L^{2}(\T)}=\|\widehat{f}(\la_{b, m, N}^h-\la_{b, m, N})\|_{\ell^{2}(\Z)}\\
   \nonumber &\le \|\la_{b, m, N}^h-\la_{b, m, N}\|_{\ell^{\8}(\Z)}\|\widehat{f}\|_{\ell^2(\Z)}\\
   \nonumber &=
   \|\la_{b, m, N}^h-\la_{b, m, N}\|_{\ell^{\8}(\Z)}\|f\|_{L^2(\T)}\\
   \nonumber&\le\big(\|\la_{b, m, N}^h\|_{\ell^{\8}(\Z)}+\|\la_{b, m, N}\|_{\ell^{\8}(\Z)}\big)\|f\|_{L^2(\T)}\lesssim \frac{\log^2 N}{\vp(N)}\|f\|_{L^2(\T)}.
 \end{align}
 On the other hand, we see that
 \begin{align*}
   \mathcal{F}_{\Z}[\la_{b, m, N}^h-\la_{b, m, N}](\xi)=\frac{\phi(m)}{mN}\bigg(\sum_{\genfrac{}{}{0pt}{}{p\in[b, mN+b]\cap\mathbf{P}_h}{p\equiv b(\mathrm{mod} m)}}\vp'(p)^{-1}\log p\ e^{2\pi i \xi p}-\sum_{\genfrac{}{}{0pt}{}{p\in[b, mN+b]\cap\mathbf{P}}{p\equiv b(\mathrm{mod} m)}}\log p\ e^{2\pi i \xi p}\bigg).
 \end{align*}
 Therefore, Lemma \ref{formlem} yields that
 \begin{align}\label{ln}
   \|f*\mathcal{F}_{\Z}[\la_{b, m, N}^h-\la_{b, m, N}]\|_{L^{\8}(\T)}&\le\|\mathcal{F}_{\Z}[\la_{b, m, N}^h-\la_{b, m, N}]\|_{L^{\8}(\T)}\|f\|_{L^1(\T)}\lesssim \frac{1}{N^{\chi+\e}}\|f\|_{L^1(\T)},
 \end{align}
 for any $\chi>0$ such that $16(1-\g)+28\chi<1$ and some $\e>0$. Thus Riesz--Thorin interpolation theorem guarantees (since $\frac{1}{r}=\frac{1-\te}{2}$) that
 \begin{multline*}
   \|f*\mathcal{F}_{\Z}[\la_{b, m, N}^h-\la_{b, m, N}]\|_{L^{r}(\T)}
   \le\|\la_{b, m, N}^h-\la_{b, m, N}\|_{\ell^{\8}(\Z)}^{2/r}\cdot\|\mathcal{F}_{\Z}[\la_{b, m, N}^h-\la_{b, m, N}]\|_{L^{\8}(\T)}^{1-2/r}\cdot\|f\|_{L^{r'}(\T)}\\
   \ \ \ \ \ \ \ \ \ \ \ \ \ \lesssim \left(\frac{\log^2 N}{\vp(N)}\right)^{2/r}\cdot\left(\frac{1}{N^{\chi+\e}}\right)^{1-2/r}\|f\|_{L^{r'}(\T)}
   \lesssim N^{-2/r}\cdot \left(\frac{1}{N^{\g-\d-1}}\right)^{2/r}\cdot\left(\frac{1}{N^{\chi+\e}}\right)^{1-2/r}\|f\|_{L^{r'}(\T)},
 \end{multline*}
for appropriately small $\d>0$, since $x^{\g-\e_1}\lesssim_{\e_1}\vp(x)$ and $\log x\lesssim_{\e_2} x^{\e_2}$ for suitable choice of $\e_1, \e_2>0$. Thus it remains to verify that
$2(\g-\d-1)/r+(1-2/r)(\chi+\e)>0\Longleftrightarrow(r-2)(\chi+\e)/2>1-\g+\d$. If $\g=1$ there is nothing to do, we take $0<\d<(r-2)(\chi+\e)/2$. If $\g\in(15/16, 1)$ then it suffices to take $\chi=\frac{2(1-\g)}{r-2}>0$ and $0<\d<\frac{\e(r-2)}{2}$, since
\begin{align*}
16(1-\g)&+28\chi<1\Longleftrightarrow 16(1-\g)(r-2)+56(1-\g)<r-2\\
&\Longleftrightarrow16r(1-\g)+24(1-\g)<r-2\Longleftrightarrow\frac{2+24(1-\g)}{1-16(1-\g)}<r
\Longleftrightarrow\frac{26-24\g}{16\g-15}<r,
\end{align*}
and the proof of Theorem \ref{bgthm1} is completed.
\end{proof}
Now we finish this section by proving Theorem \eqref{HLthm}.
\begin{proof}[Proof of Theorem \ref{HLthm}] Let $(a_n)_{n\in\N}$ be a sequence of complex numbers such that
$|a_n|\le1$ for any $n\in\N$. It suffices to use Theorem \ref{bgthm1} with $m=1$, $b=0$ and $f(n)=\frac{a_n\vp'(n)}{\log n}$. Then for any $r>\frac{26-24\g}{16\g-15}$ we have
\begin{align*}
 \int_{\T}\bigg|\sum_{p\in\mathbf{P}_{h, N}}f(p)\vp'(p)^{-1}\log p\ e^{2\pi i \xi p}\bigg|^rd\xi
 \lesssim_rN^{r/2-1}\bigg(\sum_{p\in\mathbf{P}_{h, N}}f(p)^2\vp'(p)^{-1}\log p\bigg)^{r/2}.
\end{align*}
Thus
\begin{align*}
 \int_{\T}\bigg|\sum_{p\in\mathbf{P}_{h, N}}a_p\ e^{2\pi i \xi p}\bigg|^rd\xi
 \lesssim_rN^{r/2-1}\bigg(\sum_{p\in\mathbf{P}_{h, N}}\frac{\vp'(p)}{\log p}\bigg)^{r/2}\lesssim_r \frac{1}{N}\left(\frac{\vp(N)}{\log N}\right)^r,
\end{align*}
since summation by parts implies that
\begin{align*}
  \sum_{p\in\mathbf{P}_{h, N}}\frac{\vp'(p)}{\log p}&\lesssim\frac{\vp'(N)\vp(N)}{\log^2 N}+\int_2^N\frac{\vp(x)}{\log x}\frac{|x^2\vp''(x)\log x-x\vp'(x)|}{x^2\log^2x}dx\\
  &\lesssim\frac{\vp(N)^2}{N\log^2 N}+\frac{\vp(N^{\e})^2}{\log^2N}+\frac{\vp(N)^2}{N\log^2N}\int_{N^{\e}}^N\frac{dx}{x\log x}\lesssim\frac{\vp(N)^2}{N\log^2 N},
\end{align*}
for sufficiently small $\e>0$. Finally, it is not difficult to see that
\begin{align*}
  \int_{\T}\bigg|\sum_{p\in\mathbf{P}_{h, N}}e^{2\pi i \xi p}\bigg|^rd\xi\gtrsim\int_{|\xi|\le1/(100N)}\bigg|\sum_{p\in\mathbf{P}_{h, N}}e^{2\pi i \xi p}\bigg|^rd\xi\gtrsim \frac{1}{N}\left(\frac{\vp(N)}{\log N}\right)^r.
\end{align*}
This completes the proof.
\end{proof}

\section{Proof of Theorem \ref{Roththm}}\label{sectroth}
In this section our main result will be proved. The scheme of the proof is similar in spirit to Green's proof \cite{G}. We encourage the reader to compare this section with Section 6 form \cite{G}. However, due to some
technical differences we will present all the details.  First of all we prove a transference principle which allows us to throw our problem to positive integers, after that we will make use of the restriction theorem for the set $\mathbf{P}_h$ -- see Theorem \ref{bgthm1}, and finally, thanks to Sanders's refinements of Roth theorem \cite{San2}, we conclude the proof. Throughout this section we will assume that $c\in[1, 72/71)$, $\g=1/c$, $h\in\mathcal{F}_c$ and $\vp$ is the inverse function to $h$. As in Section \ref{sectres}, $r>\frac{26-24\g}{16\g-15}$ and $r'$ denotes the conjugate exponent to $r>1$.
\subsection{Transference principle} Here we give a general principle which permits us to transfer our problem to $\Z_N=\Z/N\Z$. Before we do that we need the following.
\begin{lem}\label{lemroth0}
Assume that $A_0\subseteq\mathbf{P}_{h}$ and $\limsup_{n\to\8}\frac{\log n}{\vp(n)}|A_0\cap\mathbf{P}_{h, n}|>0$, then
\begin{align*}
  \limsup_{n\to\8}\frac{|A_0\cap\mathbf{P}_{h, n, 2n}|\log n }{\vp(n)}>0,
\end{align*}
where $\mathbf{P}_{h, x, y}=\mathbf{P}_{h}\cap[x, y]$.
\end{lem}
\begin{proof}
If $\limsup_{n\to\8}\frac{\log n}{\vp(n)}|A_0\cap\mathbf{P}_{h, n}|>0$ then there exists $\a_0>0$ such that for infinitely many $n\in\N$ we have $|A_0\cap\mathbf{P}_{h, n}|>\a_0\frac{\vp(n)}{\log n}$. Notice that there is $n_1\in\N$ such that for every $n\ge n_1$ we have $|\mathbf{P}_{h, n}|\le\frac{2\vp(n)}{\log n}$ by \eqref{swl1}. Lemma \ref{formfunlem} yields that $\vp(x)=x^{\g}\ell_{\vp}(x)$ and for every $t>0$ $\lim_{x\to\8}\frac{\ell_{\vp}(tx)}{\ell_{\vp}(x)}=1$. Now fix $t>0$ such that $\a_0/16>t^{\g}$ and observe that there exists $n_{2, t}\in\N$ such that for every $n\ge n_{2, t}$ we have $t\ge n^{-1/2}$ and
\begin{align*}
  \vp(tn)=t^{\g}\vp(n)\frac{\ell_{\vp}(tn)}{\ell_{\vp}(n)}=t^{\g}\vp(n)\left(
  \frac{\ell_{\vp}(tn)}{\ell_{\vp}(n)}-1\right)+t^{\g}\vp(n)\le2t^{\g}\vp(n).
\end{align*}
Thus notice that $\frac{2}{\log n}\ge\frac{1}{\log tn}$ which implies that the inequality
\begin{align*}
  |A_0\cap \mathbf{P}_{h, tn, n}|\ge\a_0\frac{\vp(n)}{\log n}-|\mathbf{P}_{h, tn}|\ge
  \a_0\frac{\vp(n)}{\log n}- \frac{2\vp(tn)}{\log(tn)}\ge \a_0\frac{\vp(n)}{\log n}-8t^{\g}\frac{\vp(n)}{\log n}\ge \frac{\a_0}{2}\frac{\vp(n)}{\log n},
\end{align*}
holds for infinitely many $n\ge \max\{n_1/t, n_{2, t}\}$. Now it is easy to see that

\begin{align*}
  \sum_{1\le k\le \log(1/t)}|A_0\cap \mathbf{P}_{h, 2^{k-1}tn, 2^ktn}|\ge \frac{\a_0}{2}\frac{\vp(n)}{\log n},
\end{align*}
hence by the pigeonhole principle there is some $1\le k\le \log(1/t)$ such that
\begin{align*}
  |A_0\cap \mathbf{P}_{h, 2^{k-1}tn, 2^ktn}|\ge \frac{\a_0}{2\log(1/t)}\frac{\vp(2^{k}tn)}{\log(2^{k}tn)}.
\end{align*}
This shows that one can produce infinitely many $n\in\N$ such that
$|A_0\cap\mathbf{P}_{h, n, 2n}|>\a\frac{\vp(2n)}{\log(2n)}$ for some $\a>0$ and the proof of the lemma follows.
\end{proof}

\begin{lem}\label{lemroth2}
Assume that $c\in[1, 72/71)$ and let $\g=1/c$, $h\in\mathcal{F}_c$ and $\vp$ be its inverse. Assume that $A_0\subseteq\mathbf{P}_{h}$ has a positive relative upper density: $\limsup_{n\to\8}\frac{\log n}{\vp(n)}|A_0\cap\mathbf{P}_{h, n, 2n}|>\a_0>0$ and does not contain any arithmetic progression of length three. Then there exists a positive real number $\a$ (which may depend on $\vp$ and $\g$) and there are infinitely many primes $N\in\mathbf{P}$ with the following properties. For every such $N\in\mathbf{P}$ there exists a set $A=A_N\subseteq\{1, 2,\ldots, \lfloor N/2\rfloor\}$ and an integer $W\in[1/8\log\log N, 1/2\log\log N]$ such that
\begin{itemize}
  \item A does not contain any arithmetic progression of length three,
  \item $\lambda_{b, m, N}^h(A)\ge \a$ for some $0\le b\le m-1$ with $(b, m)=1$, where $m=\prod_{p\in\mathbf{P}_W}p$.
\end{itemize}
\end{lem}
\begin{proof}
Take any $n\in\N$ such that $\a_0>\frac{1}{\log n}$ with $|A_0\cap\mathbf{P}_{h, n/2, n}|>\frac{\a_0\vp(n)}{\log n}$. Let $W=\lfloor1/4\log\log n\rfloor$ and $m=\prod_{p\in\mathbf{P}_W}p$. Thus we have
 $m\lesssim (1/4\log\log n)^{\frac{1/4\log\log n}{\log(1/4\log\log n)}}\le (\log n)^{1/4}$. Moreover, choose
any $N\in[2n/m, 4n/m]\cap\mathbf{P}$ which is possible due to Bertrand's postulate. Now we see that
$W\in[1/8\log\log N, 1/2\log\log N]$ and
\begin{align*}
 \sum_{\genfrac{}{}{0pt}{}{b=0}{(b, m)=1}}^{m-1}\sum_{k=n/2}^{n}\mathbf{1}_{A_0\cap P_{b, m}}(k)&=|A_0\cap
 \mathbf{P}_{h, n/2, n}|-|A_0\cap[1, m-1]|\\
 &\ge \a_0\frac{\vp(n)}{\log n}-m\ge \frac{\a_0}{2}\frac{\vp(n)}{\log n},
\end{align*}
where $P_{b, m}=\{j\in\N: j\equiv b(\mathrm{mod}m)\}$. Moreover, $x\vp'(x)\simeq\vp(x)$ and $\vp(2x)\simeq\vp(x)$ by Lemma \ref{formfunlem}. Thus, there exists a finite constant $C_{\vp}>0$ such that
\begin{align*}
 \sum_{\genfrac{}{}{0pt}{}{b=0}{(b, m)=1}}^{m-1}\sum_{k=n/2}^{n}\mathbf{1}_{A_0\cap P_{b, m}}(k)
 \vp'(k)^{-1}\log k\ge C_{\vp}\a_0 n.
\end{align*}
This in turn yields
\begin{align}\label{lemrothf2}
 \sum_{\genfrac{}{}{0pt}{}{k=n/2}{k\equiv b(\mathrm{mod}m)}}^{n}\mathbf{1}_{A_0\cap P_{b, m}}(k)
 \vp'(k)^{-1}\log k\ge \frac{C_{\vp}\a_0 n}{\phi(m)},
\end{align}
for some $0\le b\le m-1$, with $(b, m)=1$. Let us define $A=\frac{1}{m}\big(A_0\cap\{\lfloor n/2\rfloor+1,\ldots, n\}-b\big)$ and observe that $A\subseteq\{1, 2,\ldots, \lfloor N/2\rfloor\}$ and does not contain any three--term arithmetic progression when considered as a subset of $\Z_N=\Z/N\Z$. Moreover, \eqref{lemrothf2} implies
\begin{align*}
\sum_{\genfrac{}{}{0pt}{}{k=0}{mk+b\in\mathbf{P}_h}}^N\mathbf{1}_{A}(k)\frac{\phi(m)\log(mk+b)}{\vp'(mk+b)}
\ge C_{\vp}\a_0 n,
\end{align*}
therefore $\la_{b, m, N}^h(A)\ge C_{\vp}\a_0 n/(mN)\ge C_{\vp}\a_0/4$. It suffices to take $\a=C_{\vp}\a_0/4>0$ and the lemma follows.
\end{proof}

\subsection{Fourier analysis on $\Z_N$ and trilinear forms} We have reduced the matters to the set of integers and we are going to show that $A$ considered as a subset of
$\Z_N=\Z/N\Z$ contains a non--trivial three--term arithmetic progression. Fourier analysis on $\Z_N$ will be invaluable here. If $f:\Z_N\to \C$ is a function, then $\mathcal{F}_{\Z_N}[f]$ denotes its Fourier transform on $\Z_N$,
\begin{align*}
  \mathcal{F}_{\Z_N}[f](\xi)=\sum_{x\in\Z_N}f(x)e^{\frac{-2\pi i \xi x}{N}}, \ \ \mbox{for any $\xi\in\Z_N$}.
\end{align*}
Since $\Z_N$ embeds naturally into $\Z$ thus it makes sense to consider $f:\Z_N\to \C$ as a function on $\Z$ and then $\mathcal{F}_{\Z_N}[f](\xi)=\mathcal{F}_{\Z}(\xi/N)$. By $\mathcal{F}_{\Z_N}^{-1}[f]$ we will denote the inverse Fourier transform of $f$ on $\Z_N$,
 \begin{align*}
  \mathcal{F}_{\Z_N}^{-1}[f](x)=\sum_{\xi\in\Z_N}f(\xi)e^{\frac{2\pi i \xi x}{N}}, \ \ \mbox{for any $x\in\Z_N$}.
\end{align*}
 It is not difficult to see that for every function $f:\Z_N\to \C$ we have the following identity
\begin{align*}
  \mathcal{F}_{\Z_N}^{-1}\big[\mathcal{F}_{\Z_N}[f]\big](x)=N\cdot f(x), \ \ \mbox{for any $x\in\Z_N$},
\end{align*}
which is called the Fourier inversion formula. The convolution of two functions $f, g:\Z_N\to\C$ is
$f*g(x)=\sum_{y\in\Z_N}f(x-y)g(y)$ for $x\in\Z_N$. Products and convolutions are related by $$\mathcal{F}_{\Z_N}[f*g](\xi)=\mathcal{F}_{\Z_N}[f](\xi)\cdot\mathcal{F}_{\Z_N}[g](\xi), \ \ \mbox{for any $\xi\in\Z_N$}.$$
Let us introduce the trilinear form
\begin{align*}
  \Lambda_3(f, g, h)=\sum_{x, d\in\Z_N}f(x)g(x+d)h(x+2d),
\end{align*}
for any $f, g, h:\Z_N\mapsto \C$. Roughly speaking, one can think that the quantity $\Lambda(\mathbf{1}_A, \mathbf{1}_A, \mathbf{1}_A)$ measures the portion of arithmetic progressions $(x, x+d, x+2d)$ in $\Z_N$ which are contained in $A$. It is easy to see that if $N$ is odd (this is always our case) then we have the identity
\begin{align}\label{fid}
  \Lambda_3(f, g, h)=N^{-1}\sum_{\xi\in\Z_N}\mathcal{F}_{\Z_N}[f](\xi)\mathcal{F}_{\Z_N}[g](-2\xi)\mathcal{F}_{\Z_N}[h](\xi).
\end{align}
 Indeed,  by the Fourier inversion formula we have
 \begin{align*}
   \Lambda_3(f, g, h)=N^{-3}\sum_{\xi_1, \xi_2, \xi_3\in\Z_N}\mathcal{F}_{\Z_N}[f](\xi_1)\mathcal{F}_{\Z_N}[g](\xi_2)
   \mathcal{F}_{\Z_N}[h](\xi_3)\Lambda_3\Big(e^{\frac{2\pi i \xi_1\cdot}{N}}, e^{\frac{2\pi i \xi_2\cdot}{N}}, e^{\frac{2\pi i \xi_3\cdot}{N}}\Big),
 \end{align*}
 and this proves \eqref{fid}, since
 \begin{align*}
   \Lambda_3\Big(e^{\frac{2\pi i \xi_1\cdot}{N}}, e^{\frac{2\pi i \xi_2\cdot}{N}}, e^{\frac{2\pi i \xi_3\cdot}{N}}\Big)=\sum_{x, d\in\Z_N}e^{\frac{2\pi i \xi_1x}{N}}e^{\frac{2\pi i \xi_2(x+d)}{N}}
   e^{\frac{2\pi i \xi_3(x+2d)}{N}}=N^2\mathbf{1}_{\{\xi_2=-2\xi_1, \xi_3=\xi_1\}}(\xi_1).
 \end{align*}

\begin{lem}\label{lemroth3}
Let $N\in\mathbf{P}$ and $W\in[1/8\log\log N, 1/2\log\log N]$ be the integers as in Lemma \ref{lemroth2}. Then for sufficiently large $N$, we have
\begin{align}\label{lemrothf3}
  \sup_{\xi\in\Z_N\setminus\{0\}}|\mathcal{F}_{\Z_N}[\la_{b, m, N}^h](\xi)|\le 4\log\log W/W.
\end{align}
\end{lem}
\begin{proof}
The proof of \eqref{lemrothf3} will be a consequence of Green's inequality (see \cite{G} Lemma 6.2)
\begin{align*}
  \sup_{\xi\in\Z_N\setminus\{0\}}|\mathcal{F}_{\Z_N}[\la_{b, m, N}](\xi)|\le 2\log\log W/W,
\end{align*}
and the identity \eqref{form}
\begin{align*}
  \sum_{\genfrac{}{}{0pt}{}{p\in\mathbf{P}_{h, N}}{p\equiv b(\mathrm{mod}m)}}\vp'(p)^{-1}\log p e^{2\pi i \xi p}=\sum_{\genfrac{}{}{0pt}{}{p\in\mathbf{P}_{N}}{p\equiv b(\mathrm{mod}m)}}\log p e^{2\pi i \xi p}+O(N^{1-\chi-\chi'}),
\end{align*}
with some $\chi>0$ and $\chi'>0$, which holds uniformly with respect to $\xi\in[0, 1]$. Indeed,
\begin{multline*}
 \sup_{\xi\in\Z_N\setminus\{0\}}|\mathcal{F}_{\Z_N}[\la_{b, m, N}^h](\xi)|\le \sup_{\xi\in\Z_N\setminus\{0\}}|\mathcal{F}_{\Z_N}[\la_{b, m, N}^h](\xi)-\mathcal{F}_{\Z_N}[\la_{b, m, N}](\xi)|+
 2\log\log W/W\\
 =\sup_{\xi\in\Z_N\setminus\{0\}}\bigg|\sum_{\genfrac{}{}{0pt}{}{0\le n\le N}{mn+b\in\mathbf{P}_{h}}}
 \frac{\phi(m)\log(mn+b)}{mN\vp'(mn+b)} e^{\frac{2\pi i \xi n}{N}}-\sum_{\genfrac{}{}{0pt}{}{0\le n\le N}{mn+b\in\mathbf{P}}}
 \frac{\phi(m)\log(mn+b)}{mN} e^{\frac{2\pi i \xi n}{N}}\bigg|+2\log\log W/W\\
 =\sup_{\xi\in\Z_N\setminus\{0\}}\bigg|\sum_{\genfrac{}{}{0pt}{}{0\le n\le N}{mn+b\in\mathbf{P}_{h}}}
 \frac{\phi(m)\log(mn+b)}{mN\vp'(mn+b)} e^{\frac{2\pi i \xi(mn+b)}{mN}}-\sum_{\genfrac{}{}{0pt}{}{0\le n\le N}{mn+b\in\mathbf{P}}}
 \frac{\phi(m)\log(mn+b)}{mN} e^{\frac{2\pi i \xi(mn+b)}{mN}}\bigg|\\
 +2\log\log W/W\lesssim N^{-\chi}+2\log\log W/W\le 4\log\log W/W,
\end{multline*}
since $W\in[1/8\log\log N, 1/2\log\log N]$ and this completes the proof of the lemma.
\end{proof}

Let us define a new measure $a$ on $\Z_N$ by setting
\begin{align*}
  a(D)=\sum_{x\in\Z_N}\mathbf{1}_{A\cap D}(x)\la_{b, m, N}^h(x), \ \ \mbox{for any $D\subseteq\Z_N$}.
\end{align*}
Then $a(\Z_N)\ge\a$. However, we need to construct another measure $a_1$ on $\Z_N$. Before we do that we have to
introduce some portion of necessary definitions. Let
\begin{align*}
  R=\{\xi\in\Z_N: |\mathcal{F}_{\Z_N}[a](\xi)|\ge \d\},
\end{align*}
for some $\d\in(0, 1)$ which will be specified later. Let $\|x\|$ denotes the distance of $x\in\R$ to the nearest integer. Write $R=\{\xi_1, \xi_2,\ldots, \xi_k\}$ with $k=|R|$ and write
\begin{align*}
  B=B(R, \e)=\left\{x\in\Z_N:\ \forall_{1\le i\le k}\left\|\frac{x\xi_i}{N}\right\|\le\e\right\},
\end{align*}
for the Bohr $\e$--neighbourhood of $R$  with $\e\in(0, 1)$ which will be chosen later.
By the pigeonhole principle one can see that $|B|\ge\e^k N$ -- see Lemma 4.20 in \cite{TV}. Set $\b(x)=|B|^{-1}\mathbf{1}_B(x)$ and define $a_1=a*\b*\b$. It is easy to see that $a_1(\Z_N)\ge\a$.

\begin{lem}\label{lemroth4}
Suppose that $\e^k\ge \log\log W/W$, then there is a finite constant $C_{\vp}\ge2$ such that $\|a_1\|_{\ell^{\8}(\Z_N)}\le C_{\vp}/N$.
\end{lem}
\begin{proof}
By the Fourier inversion formula $\mathcal{F}_{\Z_N}^{-1}\big[\mathcal{F}_{\Z_N}[f]\big](x)=Nf(x)$, and Lemma \ref{lemroth3} we have
\begin{align*}
  a_1(x)&=a*\b*\b(x)\le \la_{b, m, N}^h*\b*\b(x)\\
  &=N^{-1}\sum_{\xi\in\Z_N}\mathcal{F}_{\Z_N}[\la_{b, m, N}^h](\xi)\mathcal{F}_{\Z_N}^2[\b](\xi)e^{\frac{2\pi i \xi x}{N}}\\
  &\le N^{-1}\mathcal{F}_{\Z_N}[\la_{b, m, N}^h](0)\mathcal{F}_{\Z_N}^2[\b](0)\\
  &+N^{-1}\sup_{\xi\in\Z_N\setminus\{0\}}|\mathcal{F}_{\Z_N}[\la_{b, m, N}^h](\xi)|
  \sum_{\xi\in\Z_N\setminus\{0\}}|\mathcal{F}_{\Z_N}[\b](\xi)|^2\\
  &\lesssim N^{-1}+|B|^{-1}\sup_{\xi\in\Z_N\setminus\{0\}}|\mathcal{F}_{\Z_N}[\la_{b, m, N}^h](\xi)|\\
  &\lesssim N^{-1}+\frac{\log\log W}{W|B|}\le C_{\vp}/N,
\end{align*}
since $|B|\ge\e^k N$.
\end{proof}
The next lemma will be essential in the sequel. This is a discrete version of our restriction theorem and sometimes is called a discrete majorant property.
\begin{lem}\label{lemroth5}
Suppose that $r>\frac{26-24\g}{16\g-15}$. Then there is a finite constant $C_{r, \g}'>0$ such that
\begin{align*}
  \|\mathcal{F}_{\Z_N}[a]\|_{\ell^r(\Z_N)}^r\le C_{r, \g}'.
\end{align*}
\end{lem}
\begin{proof}
We shall use Theorem \ref{bgthm1} from the previous section. Then the operator $T_hf=\mathcal{F}_{\Z}[f\la_{b, m, N}^h]$ obeys the inequality   $\|T_hf\|_{L^r(\mathbb{T})}\le C_{r, \g}N^{-1/r}\|f\|_{\ell^2(\Lambda_{b, m, N}^h, \la_{b, m, N}^h)}$ for any $r>\frac{26-24\g}{16\g-15}$. This shows that
\begin{align*}
  \|\mathcal{F}_{\Z_N}[a]\|_{\ell^r(\Z_N)}^r&=\sum_{\xi\in\Z_N}|\mathcal{F}_{\Z_N}[a](\xi)|^r
  =\sum_{\xi=0}^{N-1}|\mathcal{F}_{\Z}[a](\xi/N)|^r\lesssim_{r, \g}N\int_{\mathbb{T}}|\mathcal{F}_{\Z}[a](\xi)|^rd\xi\\
  &=N\int_{\mathbb{T}}|\mathcal{F}_{\Z}[\mathbf{1}_A\la_{b, m, N}^h](\xi)|^rd\xi\lesssim_{r, \g}\|\mathbf{1}_A\|^r_{\ell^2(\Lambda_{b, m, N}^h, \la_{b, m, N}^h)}\le C_{r, \g}',
\end{align*}
where the first inequality follows from Marcinkiewicz--Zygmund theorem -- see Lemma 6.5 in \cite{G}.
\end{proof}
\subsection{Estimates for the trilinear form and completing the proof}
   If $A$ has no proper arithmetic progressions of length $3$, then the only progressions $(x, x+d, x+2d)$ which can lie in $A$ are those for which $x\in A$ and $d=0$, hence
 \begin{align}\label{arithm1}
  \Lambda_3(a, a, a)&=\sum_{x, d\in\Z_N}a(x)a(x+d)a(x+2d)=\sum_{x\in\Z_N}a(x)^3\\
 \nonumber &\le\sum_{x\in\Z_N}\la_{b, m, N}^h(x)^3\lesssim \frac{N\log^6 N}{\vp(N)^3}\lesssim\frac{1}{N^{3\g-9\e_1-1}}\lesssim \frac{1}{N^{3/2}},
\end{align}
since $\g>71/72>5/6$ and $x^{\g-\e_1}\lesssim_{\e_1}\vp(x)$ for any $\e_1>0$.
\begin{lem}\label{lemroth6}
For any $r>\frac{26-24\g}{16\g-15}$, there is a finite constant $C_1>0$ such that we have the following upper bound
\begin{align}\label{lemroth60}
  \Lambda_3(a_1, a_1, a_1)\le C_1N^{-3/2}+C_1N^{-1}\big(\e^2\d^{-r}+\d^{2-r/r'}\big).
\end{align}
\end{lem}
\begin{proof}
Let us recall that $\big|\mathcal{F}_{\Z_N}[\b](\xi)^4\mathcal{F}_{\Z_N}[\b](-2\xi)^2-1\big|\le 2^{12}\e^2$ for every $\xi\in R$ -- the proof can be found in \cite{G} Lemma 6.7. By \eqref{arithm1} and the identity \eqref{fid} we have
\begin{align*}
  \Lambda_3(a_1, a_1, a_1)&\le  \Lambda_3(a_1, a_1, a_1)- \Lambda_3(a, a, a)+CN^{-3/2}\\
  &=CN^{-3/2}+N^{-1}\sum_{\xi\in\Z_N}\mathcal{F}_{\Z_N}[a](\xi)^2\mathcal{F}_{\Z_N}[a](-2\xi)
  \big(\mathcal{F}_{\Z_N}[\b](\xi)^4\mathcal{F}_{\Z_N}[\b](-2\xi)^2-1\big),
\end{align*}
Firstly observe that, if $\g>71/72$ then $2<\frac{26-24\g}{16\g-15}<3$. Thus for any  $r\in \big(\frac{26-24\g}{16\g-15}, 3\big)$ we have
\begin{align*}
  \bigg|\sum_{\xi\in R}\mathcal{F}_{\Z_N}[a](\xi)^2\mathcal{F}_{\Z_N}[a](-2\xi)
  \big(\mathcal{F}_{\Z_N}[\b](\xi)^4\mathcal{F}_{\Z_N}[\b](-2\xi)^2-1\big)\bigg|\le 2^{12}\e^2|R|\le C\e^2\d^{-r},
\end{align*}
where the last inequality follows from Lemma \ref{lemroth5} with $r\in \big(\frac{26-24\g}{16\g-15}, 3\big)$. Indeed,
\begin{align*}
  \d^{r}|R|\le \sum_{\xi\in R}|\mathcal{F}_{\Z_N}[a](\xi)|^r\le  \sum_{\xi\in\Z_N}|\mathcal{F}_{\Z_N}[a](\xi)|^r\le C_{r, \g}'.
\end{align*}
Secondly, notice that $1<r'<2$ and $1<\frac{r}{r'}=r-1<2$, since $2<r<3$. Thus again by Lemma \ref{lemroth5} with $r\in \big(\frac{26-24\g}{16\g-15}, 3\big)$, we have
\begin{multline*}
  \bigg|\sum_{\xi\not\in R}\mathcal{F}_{\Z_N}[a](\xi)^2\mathcal{F}_{\Z_N}[a](-2\xi)
  \big(1-\mathcal{F}_{\Z_N}[\b](\xi)^4\mathcal{F}_{\Z_N}[\b](-2\xi)^2\big)\bigg|\\
  \le 2\sup_{\xi\not\in R}|\mathcal{F}_{\Z_N}[a](\xi)|^{2-r/r'}\bigg(\sum_{\xi\in\Z_N}
  \big(|\mathcal{F}_{\Z_N}[a](\xi)|^{r/r'}\big)^{r'}\bigg)^{1/r'}
  \bigg(\sum_{\xi\in\Z_N}|\mathcal{F}_{\Z_N}[a](\xi)|^{r}\bigg)^{1/r}\\
  \le 2\d^{2-r/r'}\sum_{\xi\in\Z_N}|\mathcal{F}_{\Z_N}[a](\xi)|^{r}\le C\d^{2-r/r'}.
\end{multline*}
This completes the proof of Lemma \ref{lemroth6}.
\end{proof}
The next lemma will provide a lower bound on $\Lambda_3(a_1, a_1, a_1)$. In the proof we will follow the argument pioneered by Varnavides \cite{Var} to get this bound.
\begin{lem}\label{lemroth7}
There are absolute constants $C_2, C_3>0$ such that
\begin{align}\label{lemroth70}
  \Lambda_3(a_1, a_1, a_1)\ge C_2N^{-1}e^{-C_3\a^{-1}\log^5(1/\a)}.
\end{align}
\end{lem}
\begin{proof}
Recall that Sanders's result on three--term arithmetic progressions in the integers \cite{San2} guarantees that there is a constant $B_1>0$ such that if
\begin{align*}
  M\ge e^{B_1\a^{-1}\log^5(1/\a)},
\end{align*}
then any subset of $\{1, 2,\ldots, M\}$ of density at least $\a/4C_{\vp}$ contains a non--trivial three--term arithmetic progression. Let $A'=\{x\in\Z_N: a_1(x)\ge \a/C_{\vp}N\}$, where $C_{\vp}\ge2$ is the constant from Lemma \ref{lemroth4}. Thus by Lemma \ref{lemroth4} we have
\begin{align*}
  \a\le\sum_{x\in\Z_N}a_1(x)\le\frac{C_{\vp}|A'|}{N}+\frac{\a}{C_{\vp}N}(N-|A'|),
\end{align*}
which implies that $|A'|\ge \a N/2C_{\vp}$. Let $Z$ denote the number of three--term arithmetic progressions in $A'$. It is clear that
\begin{align}\label{lemroth71}
  \sum_{x, d\in\Z_N}a_1(x)a_1(x+d)a_1(x+2d)\ge\a^3Z/C_{\vp}^3N^3.
\end{align}
We will find a lower bound for $Z$. Let $P_{a, d}=\{a, a+d, \ldots,a+(M-1)d\}$ be an arithmetic progression of length $M$ in $\Z_N$, where $a, d\in\Z_N$, $d\not=0$ and $M\le N$. If $A'\cap P_{a, d}\subseteq\Z_N$ has at least $\a M/4C_{\vp}$ elements then Sanders's theorem yields the existence at least one non--trivial arithmetic progression of length three. Fix $d\not=0$ and observe that
\begin{align*}
  \sum_{a\in\Z_N}|A'\cap P_{a, d}|=M|A'|\ge \a MN/2C_{\vp},
\end{align*}
since there are exactly $N$ different arithmetic progressions (with the difference $d\not=0$) of length $M$ in $\Z_N$ and thus each element of $A'$ is contained in exactly $M$ of them. Now we see that
\begin{align*}
  \frac{\a MN}{2C_{\vp}}\le \sum_{a\in\Z_N}|A'\cap P_{a, d}|=\sum_{a\in\Z_N:\ |A'\cap P_{a, d}|\ge \a M/4C_{\vp}}|A'\cap P_{a, d}|
  +\sum_{a\in\Z_N:\ |A'\cap P_{a, d}|< \a M/4C_{\vp}}|A'\cap P_{a, d}|,
\end{align*}
which in turn implies that
\begin{align*}
  \a MN/4C_{\vp}\le \sum_{a\in\Z_N:\ |A'\cap P_{a, d}|\ge \a M/4C_{\vp}}|A'\cap P_{a, d}|\le |\{a\in\Z_N:\ |A'\cap P_{a, d}|\ge \a M/4C_{\vp}\}|M.
\end{align*}
We have just shown that the inequality $|A'\cap P_{a, d}|\ge \a M/4C_{\vp}$ holds for at least $\a N/4C_{\vp}$ values of $a\in\Z_N$. Therefore, there are at least $\a N^2/4C_{\vp}$ arithmetic progressions $P_{a, d}$ for which $|A'\cap P_{a, d}|\ge \a M/4C_{\vp}$, whence, as we said above, Sanders's result allows us to find at least one non--trivial arithmetic progression of length three in $|A'\cap P_{a, d}|$. Each non--trivial arithmetic progression of length three in $\Z_N$ can be contained in at most $M^2$ arithmetic progressions $P_{a, d}$. Hence, when we count the arithmetic progressions of length three in $A'\cap P_{a, d}$ we are counting each such progression at most $M^2$ times. Thus we have shown that
\begin{align}\label{lemroth72}
  Z\ge\frac{\a N^2}{4C_{\vp}M^2}.
\end{align}
Taking $M=\big\lceil e^{B_1\a^{-1}\log^5(1/\a)}\big\rceil$ and combining \eqref{lemroth71} with \eqref{lemroth72}, provided that $M\le N$, we see that
\begin{align*}
  \Lambda_3(a_1, a_1, a_1)\ge\a^3Z/C_{\vp}^3N^3\ge \frac{\a^4}{8C_{\vp}^4Ne^{2B_1\a^{-1}\log^5(1/\a)}}\ge C_2N^{-1}e^{-C_3\a^{-1}\log^5(1/\a)}.
\end{align*}
If $M>N$ the bound \eqref{lemroth70} is trivial since $Z$ always contains trivial arithmetic progression. This completes the proof of the lemma.
\end{proof}
\begin{proof}[Proof of Theorem \ref{Roththm}]
Now we gathered all ingredients necessary to conclude Theorem \ref{Roththm}. Indeed, combining \eqref{lemroth60} and \eqref{lemroth70} we see that for some $C>0$ we have
\begin{align}\label{thmrothe1}
 e^{-C_3\a^{-1}\log^5(1/\a)}\le CN^{-1/2}+C\e^2\d^{-r}+C\d^{2-r/r'},
\end{align}
for any $\g>71/72$ and $r\in \big(\frac{26-24\g}{16\g-15}, 3\big)$. Our task now is to show that there are constants $C_4>0$, and $C_5>0$ such that if we take
\begin{align*}
\d=e^{-C_4\a^{-1}\log^5(1/\a)} \ \ \mbox{and}\ \ \ \e=e^{-C_5\a^{-1}\log^5(1/\a)},
\end{align*}
then \eqref{thmrothe1} is impossible and this will have contradicted to the assumption that $A$ does not contain any arithmetic progression of length three. Rewriting \eqref{thmrothe1} we obtain
\begin{align*}
 e^{-C_3\a^{-1}\log^5(1/\a)}\le CN^{-1/2}+Ce^{-(2C_5-rC_4)\a^{-1}\log^5(1/\a)}+Ce^{-C_4(2-r/r')\a^{-1}\log^5(1/\a)},
\end{align*}
thus
\begin{align*}
 e^{-C_3\a^{-1}\log^5(1/\a)}\big(1-Ce^{-(2C_5-rC_4-C_3)\a^{-1}\log^5(1/\a)}-Ce^{-(C_4(2-r/r')-C_3)\a^{-1}\log^5(1/\a)}\big)\le CN^{-1/2},
\end{align*}
It is enough to take $C_4, C_5>0$ such that
$$Ce^{-(2C_5-rC_4-C_3)\a^{-1}\log^5(1/\a)}\le1/4,\ \ \ \mbox{and}\ \ \ \ Ce^{-(C_4(2-r/r')-C_3)\a^{-1}\log^5(1/\a)}\le 1/4,$$
then
\begin{align}\label{thmrothe2}
 e^{-C_3\a^{-1}\log^5(1/\a)}\le 2CN^{-1/2}.
\end{align}
We know  that $\e^k\ge\log\log W/W$ and  $k\le C\d^{-r}$ by Lemma \ref{lemroth6}, thus $\d>0$ and $\e>0$ must satisfy $\e^{C\d^{-r}}\ge\log\log W/W$. In other words
\begin{align}\label{thmrothe3}
  Ce^{rC_4\a^{-1}\log^5(1/\a)}\cdot C_5\a^{-1}\log^5(1/\a)\lesssim \log\left(\frac{\log\log N}{\log\log\log\log N}\right).
\end{align}
Taking
\begin{align*}
  \a\ge C'\frac{(\log\log\log\log\log N)^6}{\log\log\log\log N},
\end{align*}
for some $C'>0$, we easily see that \eqref{thmrothe3} is satisfied for sufficiently large $N$, but we have a contradiction with \eqref{thmrothe2}. This completes the proof of Theorem \ref{Roththm}.
\end{proof}
\section{Estimates for some exponential sums}\label{sectexp}
The task now is to show the estimate \eqref{finbound} which will be the main ingredient in the proof
of Lemma \ref{formlem} and allows us to gain a suitable error term in \eqref{form}. Our proof will be based on  Vaughan's trick (see Lemma \ref{Vaug}) and on  Vinogradov's ideas from the ternary Goldbach problem. See for instance \cite{Nat} or \cite{GK}. However, we only touch on a few aspects of Vinogradov's theory and instead of Weyl's type estimates we will use Van der Corput's inequality (see Lemma \ref{vdc}). In order to get a better understanding of the estimate \eqref{finbound} we refer the reader to Section \ref{sectformlem}, where its need naturally arises. Throughout the last two sections we assume that $c\in[1, 16/15)$, $\g=1/c$, $h\in\mathcal{F}_c$ and $\vp$ is the inverse function to $h$.
\begin{lem}\label{finboundlem}
Assume that $P\ge1$, $\xi\in[0, 1]$ and $M=P^{1+\chi+\e}\vp(P)^{-1}$ with  $\chi>0$  such that $16(1-\g)+28\chi<1$ and $0<\e<\chi/100$. Let $q\in\N$ and $0\le a\le q-1$ such that $(a, q)=1$ and define $\Lambda_{a, q}(k)=\Lambda(k)\mathbf{1}_{P_{a, q}}(k)$ where $P_{a, q}=\{j\in\N: j\equiv a(\mathrm{mod}q)\}$.  Then for every $0< |m|\le M$ we have
\begin{align}\label{finbound}
  \bigg|\sum_{P<k\le P_1\le 2P}\Lambda_{a, q}(k)e^{2\pi i(\xi k+m\vp(k))}\bigg|&\lesssim |m|^{1/2}\log^2 P_1\ \s(P_1)^{-1/2}\vp(P_1)^{1/2} P_1^{3/8}\\
  \nonumber  &+ |m|^{1/6}\log^{6}P_1\ \s(P_1)^{-1/6}\vp(P_1)^{-1/6}P_1^{13/12}.
\end{align}
If $c>1$ then the function $\s$ is constantly equal to $1$.
\end{lem}
The proof of Lemma \ref{finboundlem} falls naturally into the scheme based on Vaughan's identity from Lemma \ref{Vaug}, which permits us to split the sum from \eqref{finbound} into four sums simpler to deal with. We are going to describe this procedure in the proof of Lemma \ref{finboundlem}.
\begin{proof}
It is easy to see that
\begin{align*}
  \mathbf{1}_{P_{a, q}}(k)=\frac{1}{q}\sum_{s=0}^{q-1}e^{\frac{2\pi i s(k-a)}{q}}=\left\{ \begin{array} {ll}
1, & \mbox{if $k\equiv a(\mathrm{mod}q)$,}\\
0, & \mbox{otherwise.}
\end{array}
\right.
\end{align*}
This implies that
\begin{align*}
 \sum_{P<k\le P_1\le 2P}\Lambda_{a, q}(k)e^{2\pi i(\xi k+m\vp(k))}=
 \frac{1}{q}\sum_{s=0}^{q-1}e^{-2\pi i sa/q}\sum_{P<k\le P_1\le 2P}\Lambda(k)e^{2\pi i((\xi+s/q) k+m\vp(k))}.
\end{align*}
In view of this identity it suffices to establish the bounds for $0<m\le M$
\begin{align}\label{finbound1}
  \bigg|\sum_{P<k\le P_1\le 2P}\Lambda(k)e^{2\pi i(\a k+m\vp(k))}\bigg|&\lesssim m^{1/2}\log^2 P_1\ \s(P_1)^{-1/2}\vp(P_1)^{1/2} P_1^{3/8}\\
  \nonumber  &+ m^{1/6}\log^{6}P_1\ \s(P_1)^{-1/6}\vp(P_1)^{-1/6}P_1^{13/12},
\end{align}
uniformly with respect to $\a=\xi+s/q$ where $1\le s<q$ and $\xi\in[0, 1]$. According to Lemma \ref{Vaug} with $v=w=\vp(P_1)P_1^{-5/8}$,  we immediately see that
\begin{align}\label{vsplit}
 \sum_{P<n\le P_1\le 2P}\Lambda(n)&e^{2\pi i(\a n+m\vp(n))}=
 \sum_{l\le v}\sum_{P/l<k\le P_1/l}\log k\ \mu(l)e^{2\pi i(\a kl+m\vp(kl))}\\
\nonumber &-\bigg(\sum_{l\le v}+\sum_{v<l\le v^2}\bigg)\sum_{P/l<k\le P_1/l}\Pi_{v}(l)e^{2\pi i(\a kl+m\vp(kl))}\\
 \nonumber &+\sum_{v<l\le P_1/v}\sum_{\genfrac{}{}{0pt}{}{P/l< k\le P_1/l}{k>v}}\Lambda(k)\Xi_v(l)e^{2\pi i(\a kl+m\vp(kl))}=S_1-S_{21}-S_{22}+S_3,
\end{align}
with $\Pi_{v}(l)=\Pi_{v, v}(l)$ and $\Xi_v(l)$ which have been defined in \eqref{pixi}.

We are reduced to estimate the sums $S_1, S_{21}, S_{22}$ and $S_3$. The proof of \eqref{finbound} is completed by showing that
\begin{align}\label{ineq1}
  |S_1|,\ |S_{21}|\lesssim m^{1/2}\log^2 P_1\ \s(P_1)^{-1/2}\vp(P_1)^{1/2} P_1^{3/8},
\end{align}
and
\begin{align}\label{ineq2}
 |S_{22}|,\ |S_{3}|\lesssim m^{1/6}\log^{6}P_1\ \s(P_1)^{-1/6}\vp(P_1)^{-1/6}P_1^{13/12}.
\end{align}
The proofs of \eqref{ineq1} and \eqref{ineq2} have been carried over into the next two subsections.
\end{proof}
Before we derive the inequalities \eqref{ineq1} and \eqref{ineq2} we need the following.
\begin{lem}\label{vdclem2}
For every $m\in\Z\setminus\{0\}$, $l\in\N$, $j\ge0$ and $X\ge 1$ we have
\begin{align}\label{vdcest3}
  \bigg|\sum_{1\le k\le X}\ e^{2\pi i(\a jkl+m\vp(kl))}\bigg|\lesssim |m|^{1/2}\log(lX)\ lX\big(\s(lX)\vp(lX)\big)^{-1/2}.
\end{align}
If $c>1$ then $\s$ is constantly equal to $1$.
\end{lem}
This lemma is essential for us and will be applied repeatedly in the sequel with $j=0$ or $1$.
\begin{proof}
Let $U_{j, l}(X)$ denotes the sum in \eqref{vdcest3}, however it will be more handy to work with its dyadic counterpart. For this purpose, one  splits  $U_{j, l}(X)$ into $\log X$ dyadic pieces which have the following form $\sum_{Y<k\le Y'\le 2Y}e^{2\pi i(\a jkl+m\vp(kl))}$, where $Y\in[1, X]$. We have just reduced the matters to find an upper bound for the last sum. We may assume, without loss of generality, that $m>0$ and let $F(t)=\a jlt+m\vp(lt)$ for $t\in[Y, 2Y]$.
If $c>1$ then $t^2\vp''(t)=\vp(t)(\g+\te_1(t))(\g-1+\te_2(t))$ and
$$|F''(t)|=|ml^2\vp''(lt)|\simeq|ml^2\vp''(lY)|\simeq ml^2\frac{\vp(lY)}{(lY)^2}.$$
If $c=1$ then $t^2\vp''(t)=\vp(t)(\g+\te_1(t))\s(t)\t(t)$ and
$$|F''(t)|=|ml^2\vp''(lt)|\simeq\frac{ml^2\s(lY)\vp(lY)}{(lY)^2}.$$
Thus by Lemma \ref{vdc}
we obtain (if $c>1$ one can think that $\s$ is constantly equal to $1$)
\begin{align*}
 \bigg|\sum_{Y<k\le Y'\le 2Y}e^{2\pi i(\a klq+m\vp(kl))}\bigg|&\lesssim Y \left(\frac{ml^2\s(lY)\vp(lY)}{(lY)^2}\right)^{1/2}+\left(\frac{(lY)^2}{ml^2\s(lY)\vp(lY)}\right)^{1/2}\\
 &\lesssim m^{1/2}lY\big(\s(lY)\vp(lY)\big)^{-1/2}.
\end{align*}
Finally we obtain that
\begin{align*}
   |U_{j, l}(X)|\lesssim\log X\sup_{Y\in[1, X]}m^{1/2}lY\big(\s(lY)\vp(lY)\big)^{-1/2}
\lesssim m^{1/2}\log(lX)\ lX\big(\s(lX)\vp(lX)\big)^{-1/2},
 \end{align*}
 since $x\mapsto x\big(\s(x)\vp(x)\big)^{-1/2}$ is increasing. The proof of Lemma \ref{vdclem2} follows.
\end{proof}
\subsection{The estimates for $S_1$ and $S_{21}$}
Let $U_l(x)=\sum_{P/l\le k\le x}\ e^{2\pi i(\a lk+m\vp(lk))}$. Applying summation by parts to the inner sum in $S_1$ we see that
\begin{align*}
  S_1=\sum_{l\le v}\mu(l)\sum_{P/l<k\le P_1/l}\log k e^{2\pi i(\a kl+m\vp(kl))}
  =\sum_{l\le v}\mu(l)\bigg(U_l(P_1/l)\log(P_1/l)-\int_{P/l}^{P_1/l}U_l(x)\frac{dx}{x}\bigg).
\end{align*}
This gives
\begin{align*}
  |S_1|\le\log P_1\ \sum_{l\le v}\sup_{P/l\le x\le P_1/l}|U_l(x)|.
\end{align*}
In a similar way (having in mind that $v=\vp(P_1)P_1^{-5/8}$) we get
\begin{align*}
  |S_{21}|\le\sum_{l\le v}|\Pi_v(l)||U_l(P_1/l)|
 \lesssim \log P_1\ \sum_{l\le v}|U_l(P_1/l)|,
\end{align*}
since $|\Pi_v(l)|\le\sum_{k|l}\Lambda(k)=\log l$. Now Lemma \ref{vdclem2} applied to $U_l(x)$
allows us to conclude that
\begin{align*}
  |S_1|,\ |S_{21}| \le\log P_1\ \sum_{l\le v}\sup_{P/l\le x\le P_1/l}|U_l(x)|&\lesssim
\log P_1\ \sum_{l\le v}\sup_{P/l\le x\le P_1/l}|m|^{1/2}\log(lx)\ lx\big(\s(lx)\vp(lx)\big)^{-1/2}\\
&\lesssim
\vp(P_1)P_1^{-5/8}\log^2 P_1\ |m|^{1/2}P_1\big(\s(P_1)\vp(P_1)\big)^{-1/2}\\
&= |m|^{1/2}\log^2 P_1\ \s(P_1)^{-1/2}\vp(P_1)^{1/2} P_1^{3/8}.
\end{align*}
In the third inequality we have used the fact that the function $x\mapsto x\big(\s(x)\vp(x)\big)^{-1/2}$ is increasing. The proof of \eqref{ineq1} follows.
\subsection{The estimates for $S_{22}$ and $S_3$}
Here we shall bound $S_{22}$ and $S_3$. We start with some preliminary reductions which allow us to deal with both sums in a unified way. Similarly as for $S_1$ and $S_2$ we will be working with dyadic sums.  Observe that for $S_{22}$, we have
\begin{multline}\label{s22}
  |S_{22}|=\bigg|\sum_{v<l\le v^2}\sum_{P/l<k\le P_1/l}\Pi_{v}(l)e^{2\pi i(\a kl+m\vp(kl))}\bigg|\\
  \lesssim\log^2P_1\sup_{L\in[v, v^2]}\sup_{K\in[P/v^2, P_1/v]}\sup_{L'\in[L, 2L]}\sup_{K'\in[K, 2K]}\bigg|\sum_{L<l\le L'\le 2L}
  \sum_{\genfrac{}{}{0pt}{}{K<k\le K'\le 2K}{P<kl\le P_1}}\Pi_{v}(l)e^{2\pi i(\a kl+m\vp(kl))}\bigg|,
\end{multline}
and for $S_3$, we have
\begin{multline}\label{s3}
  |S_3|=\bigg|\sum_{v<l\le P_1/v}\sum_{\genfrac{}{}{0pt}{}{P/l< k\le P_1/l}{k>v}}\Lambda(k)\Xi_v(l)e^{2\pi i(\a kl+m\vp(kl))}\bigg|\\
 \lesssim\log^2P_1\sup_{L\in[v, P_1/v]}\sup_{K\in[v, P_1/v]}\sup_{L'\in[L, 2L]}\sup_{K'\in[K, 2K]}\bigg|
  \sum_{L<l\le L'\le 2L}\sum_{\genfrac{}{}{0pt}{}{K<k\le K'\le2K}{P< kl\le P_1}}\Lambda(k)\Xi_v(l)e^{2\pi i(\a kl+m\vp(kl))}\bigg|,
\end{multline}
where $\Pi_{v}(l)$ and $\Xi_v(l)$ are defined as in \eqref{pixi}. Now it is not difficult to observe that
\begin{align}\label{aritmfun}
\sum_{L<l\le2L}|\Pi_{v}(l)|^2\lesssim L\log^2 L, \ \ \ \ \mbox{and}\ \ \ \ \sum_{L<l\le2L}|\Xi_v(l)|^2\lesssim L\log^3L.
\end{align}
In view of these decompositions it remains to show.
\begin{lem}\label{billem}
Let $K, L\in\N$, $m\in\Z\setminus\{0\}$. Assume that $|m|\min\{K, L\}\le \s(KL)\vp(KL)$ and $\vp(KL)\le \min\{K, L\}^4$. Then
 \begin{align}\label{billem1}
  \bigg|\sum_{L<l\le L'\le 2L}\sum_{\genfrac{}{}{0pt}{}{K<k\le K'\le2K}{P< kl\le P_1}}&\Delta_1(l)\Delta_2(k)e^{2\pi i(\a kl+m\vp(kl))}\bigg|\\
\nonumber&\lesssim |m|^{1/6}\ \log^{2}L\ \log^{2}K\ \big(\s(KL)\vp(KL)\big)^{-1/6}\min\{K, L\}^{1/6}\ KL,
\end{align}
for every sequences of complex numbers $(\Delta_1(l))_{l\in(L, 2L]}$, and  $(\Delta_2(k))_{k\in(K, 2K]}$ such that
\begin{align}\label{logineq}
 \sum_{L<l\le 2L}|\Delta_1(l)|^2\lesssim L\log^{3}L,\ \ \ \mbox{and}\ \ \ \sum_{K<k\le 2K}|\Delta_2(k)|^2\lesssim K\log^{3}K.
\end{align}
\end{lem}
Assuming momentarily Lemma \ref{billem} we are in a position where we can easily derive the bounds for $S_{22}$ and $S_3$. Recall that $M=P^{1+\chi+\e}\vp(P)^{-1}$ with $\chi>0$ such that $16(1-\g)+28\chi<1$ and $0<\e<\chi/100$. The inequalities in \eqref{logineq} are satisfied with a suitable choice of $\Delta_1(l)$ and $\Delta_2(k)$ for both dyadic subsums of $S_{22}$ and $S_3$, by \eqref{aritmfun}. Observe that for sufficiently large $P_1\simeq P$ and an appropriate choice of $\e_1>0$, we have
$$P_1/v=P_1\big(\vp(P_1)P_1^{-5/8}\big)^{-1}=P_1^{13/8}\vp(P_1)^{-1}\le P_1^{13/8+\e_1-\g}\le P_1^{3/4},$$
$$P_1/v^2=P_1\big(\vp(P_1)P_1^{-5/8}\big)^{-2}=P_1^{18/8}\vp(P_1)^{-2}\ge P_1^{1/4},$$
$$v=\vp(P_1)P_1^{-5/8}\ge P_1^{\g-\e_1-5/8}\ge P_1^{1/4},\ \ \ \mbox{and}\ \ \ v^2=(\vp(P_1)P_1^{-5/8}\big)^{2}\le P_1^{3/4},$$
since $\g>15/16>7/8$. Therefore, in both cases $K, L\in[P_1^{1/4}, P_1^{3/4}]$ and $KL\simeq P_1$, hence $P_1^{1/4}\le \min\{K, L\}\le P_1^{1/2}$.
 Thus, we see that $\vp(KL)\le \min\{K, L\}^4$, if not, then $\min\{K, L\}^4<\vp(KL)\le \vp(P_1)\le P_1$, hence $\min\{K, L\}< P_1^{1/4}$ contrary to what we have just shown. Finally, it remains to verify that $|m|\min\{K, L\}\le \s(KL)\vp(KL)$. Indeed, by assumption $3/2+\chi+4\e-2\g<1/2(4(1-\g)+10\chi-1)<0$, thus
\begin{multline*}
  |m|\min\{K, L\}\le MP_1^{1/2}=P_1^{3/2+\chi+\e}\vp(P_1)^{-2}\s(P_1)^{-1}\s(P_1)\vp(P_1)\\
  \lesssim P_1^{3/2+\chi+4\e-2\g }\s(P_1)\vp(P_1)\lesssim P_1^{1/2(4(1-\g)+10\chi-1)}\s(P_1)\vp(P_1)\lesssim \s(KL)\vp(KL).
\end{multline*}

 Therefore, \eqref{billem1} yields
\begin{multline*}
\bigg|\sum_{L<l\le L'\le 2L}\sum_{\genfrac{}{}{0pt}{}{K<k\le K'\le2K}{P< kl\le P_1}}\Delta_1(l)\Delta_2(k)e^{2\pi i(\a kl+m\vp(kl))}\bigg|\\
\lesssim  |m|^{1/6}\ \log^{2}L\ \log^{2}K\ \big(\s(KL)\vp(KL)\big)^{-1/6}\ \min\{K, L\}^{1/6}\ KL\\
\lesssim |m|^{1/6}\ \log^{4}P_1\ \big(P_1^{1/2}\big)^{1/6}\ P_1\ \big(\s(P_1)\vp(P_1)\big)^{-1/6}\\
\lesssim |m|^{1/6}\log^{4}P_1\ P_1^{13/12}\vp(P_1)^{-1/6}\s(P_1)^{-1/6}.
   \end{multline*}
The proof of estimates \eqref{ineq2} is completed, since in view of the dyadic decompositions \eqref{s22} and \eqref{s3}, at the expense of $\log^2P_1$ factor we obtain
\begin{align*}
  |S_{22}|,\ |S_3|\lesssim |m|^{1/6}\log^{6}P_1\ \s(P_1)^{-1/6}\vp(P_1)^{-1/6}P_1^{13/12}.
\end{align*}

\begin{proof}[Proof of Lemma \ref{billem}]
We divide the proof into three steps. We will follow the ideas from \cite{HB} Section 5, or \cite{GK} Section 4. In the first two steps we collect necessary tools which allows us to illustrate the proof of inequality \eqref{billem1} in the third step. The symmetry between the variables $k, l$ in the sums in \eqref{billem1}  allows us to always arrange the parameters $K, L$ to satisfy $K\le L$.\\

\noindent {\textsf{\textbf{\underline{Step 1.}}}} For $r\in\Z$ define
 \begin{align*}
   E_r=\sum_{L<l\le2L}\sum_{\genfrac{}{}{0pt}{}{K<k, k+r\le K'\le 2K}{P< kl, (k+r)l\le P_1}}\Delta_2(k)\overline{\Delta_2(k+r)}e^{2\pi i (\a kl+m\vp(kl) - \a (k+r)l-m\vp((k+r)l))}.
 \end{align*}
 Notice that
 \begin{align}\label{eineq0}
   |E_0|\le\sum_{L<l\le2L}\sum_{K<k\le K'\le2K}|\Delta_2(k)|^2\lesssim L\sum_{K<k\le 2K}|\Delta_2(k)|^2\lesssim LK\log^3K.
 \end{align}
 Moreover, for any $r\in\Z\setminus\{0\}$ we have
  \begin{align*}
   E_r=\sum_{\max\{K, K-r\}<k\le \min\{K', K'-r\}}\Delta_2(k)\overline{\Delta_2(k+r)}\widetilde{S}(k, r),
 \end{align*}
where
\begin{align*}
  \widetilde{S}(k, r)=\sum_{\max\{L, \frac{P}{k}, \frac{P}{k+r}\}<l\le\min\{2L, \frac{P_1}{k}, \frac{P_1}{k+r}\}}e^{2\pi i (\a kl+m\vp(kl) - \a (k+r)l-m\vp((k+r)l))}.
\end{align*}
One can see that for every $R\ge1$ we have
\begin{multline}\label{eineq}
   \sum_{1\le|r|\le R}|E_r|\lesssim \sum_{1\le|r|\le R}\sum_{K<k, k+r\le K'}|\Delta_2(k)|^2|\widetilde{S}(k, r)|+|\overline{\Delta_2(k+r)}|^2|\widetilde{S}(k+r, -r)|\\
  \le\sum_{1\le|r|\le R}\sum_{K<k, k+r\le K'}|\Delta_2(k)|^2|\widetilde{S}(k, r)|
   +\sum_{1\le|r|\le R}\sum_{K<k, k-r\le K'}|\Delta_2(k)|^2|\widetilde{S}(k, -r)|\\
   \lesssim\sum_{1\le|r|\le R}\sum_{K<k, k+r\le K'}|\Delta_2(k)|^2|\widetilde{S}(k, r)|
   =\sum_{K<k\le K'}|\Delta_2(k)|^2\sum_{1\le|r|\le R}|\widetilde{S}(k, r)|
   \mathbf{1}_{(K, K']}(k+r),
 \end{multline}
since $|\widetilde{S}(k, r)|=|\widetilde{S}(k+r, -r)|$.\\

 \noindent {\textsf{\textbf{\underline{Step 2.}}}} We are going to show that for every $m\in\N$ and $k\in(K, 2K]$ and $R\ge1$ we have
  \begin{align}\label{vdcest4}
    \frac{1}{R}\sum_{1\le |r|\le R}|\widetilde{S}(k, r)|\mathbf{1}_{(K, 2K]}(k+r)\lesssim m^{1/2}R^{1/2}KL\big(\s(KL)\vp(KL)\big)^{-1/2}K^{-1/2}.
  \end{align}
For this purpose we will proceed likewise in Lemma \ref{vdclem2}. Let $$F(x)=\a kx+m\vp(kx) - \a (k+r)x-m\vp((k+r)x)),$$ for $x\in(L, 2L]$ and note that according to Lemma \ref{funlemfi} and the mean value theorem,  for some $\eta\in(0, 1)$ and $\eta_{k, r}=k+\eta r$ if $r>0$ and $\eta_{k, r}=k+r-\eta r$ if $r<0$, we have
\begin{multline*}
 |F''(x)|=|mk^2\vp''(kx)-m(k+r)^2\vp''((k+r)x))|\\
= \big|r\big(2m\eta_{k, r}\vp{''}(x\eta_{k, r})+m\eta_{k, r}^{2}x\vp{'''}(x\eta_{k, r})\big)\big|\\
= |rm\eta_{k, r}\vp{''}(x\eta_{k, r})(2+\b_{3}+\theta_{3}(x\eta_{k, r}))|\\
\simeq |mrK\vp{''}(KL)|\simeq \frac{m|r|K\s(KL)\vp(KL)}{(KL)^{2}},
\end{multline*}
since $k, k+r\in(K, 2K]$ and $\eta_{k, r}\in(K, 2K]$. Therefore by Lemma \ref{vdc} we obtain (as before we think that $\s$ is constantly equal to $1$, if $c>1$)
\begin{align*}
  |\widetilde{S}(k, r)|
  &\lesssim L\left(\frac{m|r|K\s(KL)\vp(KL)}{(KL)^2}\right)^{1/2}+\left(\frac{(KL)^2}{m|r|K\s(KL)\vp(KL)}\right)^{1/2}\\
  &\lesssim (m|r|L)^{1/2}+KL\big(\s(KL)\vp(KL)\big)^{-1/2}K^{-1/2}\\
  &\lesssim m^{1/2}|r|^{1/2}KL\big(\s(KL)\vp(KL)\big)^{-1/2}K^{-1/2},
\end{align*}
and \eqref{vdcest4} follows.

Therefore combining \eqref{eineq} with \eqref{vdcest4} we obtain that
\begin{align}\label{eineq1}
  \frac{1}{R}\sum_{1\le|r|\le R}|E_r|\lesssim\sum_{K<k\le K'}|\Delta_2(k)|^2\frac{1}{R}\sum_{1\le|r|\le R}|\widetilde{S}(k, r)|\mathbf{1}_{(K, K']}(k+r)\\
  \nonumber\lesssim K\log^3K\cdot m^{1/2}R^{1/2}KL\big(\s(KL)\vp(KL)\big)^{-1/2}K^{-1/2}.
\end{align}

\noindent {\textsf{\textbf{\underline{Step 3.}}}} By the Cauchy--Schwartz inequality and Lemma \ref{vdc1}, applied with $H=K$ and an integer $1\le R\le K$ which will be specified later, we immediately see that

\begin{multline}\label{eineq2}
   \bigg|\sum_{L<l\le L'\le2L}\sum_{\genfrac{}{}{0pt}{}{K<k\le K'\le 2K}{P< kl\le P_1}}\Delta_1(l)\Delta_2(k)e^{2\pi i (\a kl+m\vp(kl)}\bigg|^2\\
   \le\bigg(\sum_{L<l\le2L}|\Delta_1(l)|^2\bigg) \sum_{L<l\le L'\le 2L}\bigg|\sum_{\genfrac{}{}{0pt}{}{K<k \le K'\le 2K}{P< kl\le P_1}}\Delta_2(k)e^{2\pi i (\a kl+m\vp(kl)}\bigg|^2\\
   \lesssim L\log^{3}L \sum_{L<l\le2L}\bigg|\sum_{\genfrac{}{}{0pt}{}{K<k\le K'\le 2K}{P< kl\le P_1}}\Delta_2(k)e^{2\pi i (\a kl+m\vp(kl)}\bigg|^2\\
   \lesssim L\log^{3}L\ \frac{K+R}{R}\sum_{|r|\le R}\left(1-\frac{|r|}{R}\right)|E_r|\\
   \lesssim L^2K\log^{3}L\log^{3}K\ \frac{K+R}{R}+L\log^{3}L\ \frac{K+R}{R}\sum_{1\le|r|\le R}|E_r|\\
   \lesssim \log^{3}L\log^{3}K\left(\frac{L^2K^2}{R}+K^2L m^{1/2}R^{1/2}KL\big(\s(KL)\vp(KL)\big)^{-1/2}K^{-1/2}\right),
 \end{multline}
 where we have used the estimate \eqref{eineq0} for $|E_0|$ and the inequality \eqref{eineq1}. Now we are able to finish our proof. Taking $R=\lceil m^{-a}K^{-b}L^c\big(\s(KL)\vp(KL)\big)^{-d}\rceil$ for some $a, b, c, d\in\R$ we see that the last expression in \eqref{eineq2} is bounded by
\begin{multline*}
\log^{3}L\log^{3}K\Big(m^aK^{2+b}L^{2-c}\big(\s(KL)\vp(KL)\big)^{d}\\
+ m^{1/2-a/2}K^{5/2-b/2}L^{2+c/2}\big(\s(KL)\vp(KL)\big)^{-d/2-1/2}\Big).
   \end{multline*}
We will impose some restrictions on $a, b, c, d\in\R$ which make the last two terms equal. It suffices to take
$a=b=1/3, c=0, d=-1/3$.
We now easily see that $1\le m^{-1/3}K^{-1/3}\big(\s(KL)\vp(KL)\big)^{1/3}\le K$ by our assumptions, thus $1\le R\lesssim K$ and consequently \eqref{billem1} follows, since
\begin{multline*}
\bigg|\sum_{L<l\le L'\le2L}\sum_{\genfrac{}{}{0pt}{}{K<k\le K'\le 2K}{P< kl\le P_1}}\Delta_1(l)\Delta_2(k)e^{2\pi i (\a kl+m\vp(kl)}\bigg|\\
\lesssim  m^{1/6}\ \log^{2}L\ \log^{2}K\ \big(\s(KL)\vp(KL)\big)^{-1/6}\ K^{1/6}\ KL.
   \end{multline*}
   \end{proof}
\section{Proof of Lemma \ref{formlem}}\label{sectformlem}
This section provides a detailed proof of Lemma \ref{formlem}. We are going to follow the ideas of Heath--Brown  \cite{HB}. We shall split the proof of \eqref{form} into three steps. In the third step we will be able to use estimate carried by Lemma \ref{billem} which will turn out to be decisive there and permits us to complete the proof.
\subsection{The first reduction} We start with the following.
\begin{lem}\label{lemest0}
Let $\Phi(x)=\{x\}-1/2$ and $\Lambda(n)$ denote von Mangoldt's function as in Section 3 and $\gamma, \chi>0$ satisfy conditions from Lemma \ref{formlem}. Then for every $q\in\N$ and $0\le a\le q-1$ such that $(a, q)=1$, $N\in\N$ and $0<\e<\chi/100$ we have
\begin{align}\label{lem0eq1}
  \sum_{\genfrac{}{}{0pt}{}{p\in\mathbf{P}_{h, N}}{p\equiv a(\mathrm{mod}q)}}\vp'(p)^{-1}\log p\ e^{2\pi i \xi p}&=\sum_{\genfrac{}{}{0pt}{}{p\in\mathbf{P}_{N}}{p\equiv a(\mathrm{mod}q)}} \log p\ e^{2\pi i \xi p}\\
 \nonumber +\sum_{k=1}^N \vp'(k)^{-1}&\big(\Phi(-\vp(k+1))-\Phi(-\vp(k))\big)\Lambda_{a, q}(k)e^{2\pi i \xi k}+
  O\big(N^{1-\chi+\e}\big),
\end{align}
where $\Lambda_{a, q}(k)=\Lambda(k)\mathbf{1}_{P_{a, q}}(k)$ and $P_{a, q}=\{j\in\N: j\equiv a(\mathrm{mod}q)\}$.
\end{lem}
\begin{proof}
We shall apply Lemma \ref{filem} to the first sum in \eqref{lem0eq1}. However, we should remember that the identity from \eqref{intlemform} holds for sufficiently large $p\in\mathbf{P}_h$, say $p\ge N_0$. Therefore, we have to split the sum in \eqref{lem0eq1} into two parts, that over $p\in\mathbf{P}_{h, N_0}$ and that over $p\in\mathbf{P}_{h, N}$ with $p\ge N_0$. When $p\in\mathbf{P}_{h, N_0}$ the sum can be trivially estimated from above by $N_0$, otherwise when $p\in\mathbf{P}_{h, N}$ with $p\ge N_0$ we use Lemma \ref{filem}. Finally, we  complete the summation $p\in\mathbf{P}_{N}$ with $p\ge N_0$ in the second sum (after the application of Lemma \ref{filem}) to all $p\in\mathbf{P}_{N}$ at the expense of additional term depending on $N_0$ which is harmless, since we are only interested in large values of $N\ge N_0$. This remark shows that one can assume that the identity in \eqref{intlemform} holds for all $p\in\mathbf{P}_{h}$.

According to Lemma \ref{filem} and the definition of function $\Phi(x)=\{x\}-1/2$ we obtain that for every $p\in\N$ there exists  $\xi_p\in(0, 1)$ such that
\begin{align*}
  \lfloor-\vp(p)\rfloor-\lfloor-\vp(p+1)\rfloor&=\vp(p+1)-\vp(p)+\Phi(-\vp(p+1))-\Phi(-\vp(p))\\
  &=\vp'(p)+\vp''(p+\xi_p)/2+\Phi(-\vp(p+1))-\Phi(-\vp(p)).
\end{align*}
Thus
\begin{multline*}
  \sum_{\genfrac{}{}{0pt}{}{p\in\mathbf{P}_{h, N}}{p\equiv a(\mathrm{mod} q)}}\vp'(p)^{-1}\log p\ e^{2\pi i \xi p}=\sum_{\genfrac{}{}{0pt}{}{p\in\mathbf{P}_{N}}{p\equiv a(\mathrm{mod}q)}} \vp'(p)^{-1}\big(\lfloor-\vp(p)\rfloor-\lfloor-\vp(p+1)\rfloor\big)\log p\ e^{2\pi i \xi p}\\
  \ \ \ \ \ \ \ \ \  =\sum_{\genfrac{}{}{0pt}{}{p\in\mathbf{P}_{N}}{p\equiv a(\mathrm{mod}q)}} \log p\ e^{2\pi i \xi p}
  +\sum_{\genfrac{}{}{0pt}{}{p\in\mathbf{P}_{N}}{p\equiv a(\mathrm{mod}q)}} \vp'(p)^{-1}\big(\Phi(-\vp(p+1))-\Phi(-\vp(p))\big)\log p\ e^{2\pi i \xi p}
  +O(\log N),
\end{multline*}
since by Mertens theorem (see \cite{Nat} Theorem 6.6, page 160) we have
\begin{align*}
  O\bigg(\sum_{\genfrac{}{}{0pt}{}{p\in\mathbf{P}_{N}}{p\equiv a(\mathrm{mod}q)}} \frac{\vp''(p+\xi_p)\log p}{2\vp'(p)}\ e^{2\pi i \xi p}\bigg)=O\bigg(\sum_{p\in\mathbf{P}_N} \frac{\vp''(p)\log p}{\vp'(p)}\bigg)=
  O\bigg(\sum_{p\in\mathbf{P}_N} \frac{\log p}{p}\bigg)=O(\log N).
\end{align*}
Now observe that
\begin{align*}
  \sum_{\genfrac{}{}{0pt}{}{p\in\mathbf{P}_{N}}{p\equiv a(\mathrm{mod}q)}} &\vp'(p)^{-1}\big(\Phi(-\vp(p+1))-\Phi(-\vp(p))\big)\log p\ e^{2\pi i \xi p}\\
  &=\sum_{k=1}^N \vp'(k)^{-1}\big(\Phi(-\vp(k+1))-\Phi(-\vp(k))\big)\Lambda_{a, q}(k)e^{2\pi i \xi k}+
  O\bigg(\frac{N}{\vp(N)}\sum_{\genfrac{}{}{0pt}{}{p\in\mathbf{P}_N: 1\le p^s\le N}{s\ge2}} \log p\bigg)\\
  &=\sum_{k=1}^N \vp'(k)^{-1}\big(\Phi(-\vp(k+1))-\Phi(-\vp(k))\big)\Lambda_{a, q}(k)e^{2\pi i \xi k}+
  O\big(N^{3/2-\g+\e}\big).
\end{align*}
The last identity follows from the following observation
\begin{align*}
  O\bigg(\sum_{\genfrac{}{}{0pt}{}{p\in\mathbf{P}_N: 1\le p^s\le N}{s\ge2}} \log p\bigg)=O\bigg(\sum_{p\in\mathbf{P}_N: 1\le p^2\le N}\left\lfloor\frac{\log N}{\log p}\right\rfloor \log p\bigg)
  =O\Big(\pi\big(N^{1/2}\big)\log N\Big)=O\big(N^{1/2}\big).
\end{align*}
The proof is completed since $O\big(N^{3/2-\g+\e}\big)=O\big(N^{1-\chi-\e}\big)$. Indeed, we easily see that
$3/2-\g+\e=1-\chi-\e+(2(1-\g)-1+2(2\e+\chi))/2<1-\chi-\e$ as desired.
\end{proof}
\subsection{The second reduction}
The proof of Lemma \ref{formlem} will be completed if we show that
\begin{align}\label{forma1}
  \sum_{k=1}^N \vp'(k)^{-1}\big(\Phi(-\vp(k+1))-\Phi(-\vp(k))\big)\Lambda_{a, q}(k)e^{2\pi i \xi k}=O\big(N^{1-\chi-\chi'}\big),
\end{align}
for $\chi>0$ such that $16(1-\g)+28\chi<1$ and some $\chi'>0$.
\begin{lem}\label{lemest1}
Assume that $P\ge1$ and  $M=P^{1+\chi+\e}\vp(P)^{-1}$ with  $\chi>0$  such that $16(1-\g)+28\chi<1$ and $0<\e<\chi/100$.
Then for every $q\in\N$ and $0\le a\le q-1$ such that $(a, q)=1$ we have
\begin{align}\label{lem1eq1}
  &\sum_{P<k\le P_1\le 2P} \vp'(k)^{-1}\big(\Phi(-\vp(k+1))-\Phi(-\vp(k))\big)\Lambda_{a, q}(k)e^{2\pi i \xi k}\\
  \nonumber=\sum_{0<|m|\le M}&\frac{1}{2\pi i m}\sum_{P<k\le P_1\le 2P} \vp'(k)^{-1}\Lambda_{a, q}(k) \Big(e^{2\pi i(\xi k+m\vp(k+1))}-e^{2\pi i(\xi k+m\vp(k))}\Big)+O\big(P^{1-\chi-\e}\big),
\end{align}
where $\Lambda_{a, q}(k)=\Lambda(k)\mathbf{1}_{P_{a, q}}(k)$ and $P_{a, q}=\{j\in\N: j\equiv a(\mathrm{mod}q)\}$.
\end{lem}
\begin{proof}
Let $S$ denote the first sum in \eqref{lem1eq1}, then  the Fourier expansion \eqref{four1} of the function $\Phi$ leads us to
\begin{align*}
  S
  &=\sum_{0<|m|\le M}\frac{1}{2\pi i m}\sum_{P<k\le P_1\le 2P} \vp'(k)^{-1}\Lambda_{a, q}(k) \Big(e^{2\pi i(\xi k+m\vp(k+1))}-e^{2\pi i(\xi k+m\vp(k))}\Big)\\
  &+O\bigg(\sum_{P<k\le P_1\le 2P}\vp'(k)^{-1}\Lambda_{a, q}(k)\left(\min\left\{1, \frac{1}{M\|\vp(k)\|}\right\}+\min\left\{1, \frac{1}{M\|\vp(k+1)\|}\right\}\right)\bigg).
\end{align*}
The only point remaining concerns the behaviour of the error term with $\min\big\{1, \frac{1}{M\|\vp(k)\|}\big\}$. The same reasoning will apply to the sum with $\min\big\{1, \frac{1}{M\|\vp(k+1)\|}\big\}$ equally well. Thus by \eqref{four2} we see
\begin{align*}
 \sum_{P<k\le P_1\le 2P}\frac{\Lambda_{a, q}(k)}{\vp'(k)}\cdot\min\left\{1, \frac{1}{M\|\vp(k)\|}\right\}
 &\lesssim \frac{\log P}{\vp'(P)}\sum_{P<k\le P_1\le 2P}\sum_{m\in\Z}b_m e^{2\pi im\vp(k)}\\
 &\lesssim\frac{\log P}{\vp'(P)}\sum_{m\in\Z}|b_m|\bigg|\sum_{P<k\le P_1\le 2P} e^{2\pi im\vp(k)}\bigg|.
\end{align*}

It suffices to estimate the last sum. Namely, Lemma \ref{vdclem2} applied to the inner sum with $l=1$ and $j=0$ (in fact we refer to the proof of Lemma \ref{vdclem2})  and the bounds \eqref{fcoe2} for $|b_m|$ imply that
\begin{align*}
  \sum_{m\ge0}|b_m|\bigg|\sum_{P<k\le P_1\le 2P} e^{2\pi im\vp(k)}\bigg|\lesssim \frac{P\ \log M}{M} +
  \bigg(\sum_{0<m\le M}+\sum_{m> M}\bigg)|b_m|\frac{m^{1/2}P}{\big(\s(P)\vp(P)\big)^{1/2}}\\
  \lesssim\frac{P\ \log M}{M}+\sum_{0<m\le M}m^{1/2}\frac{\log M}{M}\frac{P}{\big(\s(P)\vp(P)\big)^{1/2}}
  +\sum_{m> M}\frac{M}{m^{3/2}}\frac{P}{\big(\s(P)\vp(P)\big)^{1/2}}\\
  \lesssim\frac{P\ \log M}{M}+\log M M^{1/2}\frac{P}{\big(\s(P)\vp(P)\big)^{1/2}}.
\end{align*}
Taking $M=P^{1+\chi+\e}\vp(P)^{-1}$,  we obtain
\begin{multline*}
\frac{\log P}{\vp'(P)} \sum_{m\ge0}|b_m|\bigg|\sum_{P<k\le P_1\le 2P} e^{2\pi im\vp(k)}\bigg|
\lesssim
 \frac{P\ \log M\ \log P}{\vp'(P)M}+\log M M^{1/2}\frac{P\log P}{\vp'(P)\big(\s(P)\vp(P)\big)^{1/2}}\\
 \lesssim \frac{\vp(P)P^{-\chi-\e}}{\vp'(P)}\ \log^2 P+ \frac{P^{3/2+\chi/2+\e/2}}{\vp'(P)\s(P)^{1/2}\vp(P)}\ \log^2 P\\
 \lesssim\frac{\vp(P)P^{-\chi-\e}}{\vp'(P)}\ \log^2 P\left(1+\frac{P^{3/2+3\chi/2+3\e/2}}{\s(P)^{1/2}\vp(P)^2}\right)
 \lesssim\frac{\vp(P)P^{-\chi-\e}}{\vp'(P)}\lesssim P^{1-\chi-\e}.
\end{multline*}
Taking $0<\e<\chi/100$ we may conclude that the expression in the last parenthesis is bounded. Indeed, due to the inequalities $x^{\g-\e_1}\lesssim_{\e_1}\vp(x)$, and $(\s(x))^{-1}\lesssim_{\e_2}x^{\e_2}$
which hold for arbitrary $\e_1, \e_2>0$ we easily see (taking $\e_1=\e_2=\e>0$) that $3/2+3\chi/2+3\e/2+\e/2-2\g+2\e<0,$ since $3+3\chi+8\e-4\g<4(1-\g)+4\chi-1<0$,
where the last inequality is obviously satisfied and this finishes the proof.
\end{proof}
\subsection{The third reduction -- completing the proof}
Now we can complete the proof of Lemma \ref{formlem}. Our main tool will be Lemma \ref{finboundlem}.
\begin{proof}[Proof of Lemma \ref{formlem}]
Recall that  $\g, \chi>0$  satisfy $16(1-\g)+28\chi<1$ and $0<\e<\chi/100$. Then combining Lemma \ref{lemest0} with Lemma \ref{lemest1} we immediately see that
\begin{multline}\label{estim1}
  \bigg|\sum_{\genfrac{}{}{0pt}{}{p\in\mathbf{P}_{h, N}}{p\equiv a(\mathrm{mod}q)}}\vp'(p)^{-1}\log p\ e^{2\pi i \xi p}-\sum_{\genfrac{}{}{0pt}{}{p\in\mathbf{P}_{N}}{p\equiv a(\mathrm{mod}q)}} \log p\ e^{2\pi i \xi p}\bigg|\\
  \lesssim\log N \sup_{1\le P\le N}\bigg|\sum_{P<k\le P_1\le 2P} \vp'(k)^{-1}\big(\Phi(-\vp(k+1))-\Phi(-\vp(k))\big)\Lambda_{a, q}(k)e^{2\pi i \xi k}\bigg|+N^{1-\chi-\e}\\
  \lesssim\log N \sup_{1\le P\le N}\sum_{0<|m|\le M}\frac{1}{m}\bigg|\sum_{P<k\le P_1\le 2P} \vp'(k)^{-1}\Lambda_{a, q}(k) \Big(e^{2\pi i(\xi k+m\vp(k+1))}-e^{2\pi i(\xi k+m\vp(k))}\Big)\bigg|\\
  +N^{1-\chi-\e}\log N,
\end{multline}
where $M=P^{1+\chi+\e}\vp(P)^{-1}$.
In order to estimate the error term in \eqref{estim1} let us define $U_m(x)=\sum_{P< k\le x}\Lambda_{a, q}(k)e^{2\pi i(\xi k+m\vp(k))}$, and  $\phi_m(k)=\vp'(k)^{-1}\big(e^{2\pi im(\vp(k+1)-\vp(k))}-1\big)$. It is easy to verify that $|\phi_m(x)|\lesssim m$ and $|\phi_m'(x)|\lesssim\frac{m}{x}$, thus summation by parts combined with the estimate \eqref{finbound} yield


\begin{multline}\label{lastest}
\sum_{0<|m|\le M}\frac{1}{m}\bigg|\sum_{P<k\le P_1\le 2P} \vp'(k)^{-1}\Lambda_{a, q}(k) \Big(e^{2\pi i(\xi k+m\vp(k+1))}-e^{2\pi i(\xi k+m\vp(k))}\Big)\bigg|\\
 \lesssim \sum_{m=1}^M\frac{1}{m}\bigg(|U_m(P_1)\phi_m(P_1)|+\int_P^{P_1}|U_m(x)\phi_m'(x)|dx\bigg)
 \lesssim \sum_{m=1}^M\sup_{x\in(P, 2P]}|U_m(x)|\\
 \lesssim
 \sum_{m=1}^M m^{1/2}\log^2 P\ \s(P)^{-1/2}\vp(P)^{1/2} P^{3/8}\\
 +
 \sum_{m=1}^M m^{1/6}\log^{6}P\ \s(P_1)^{-1/6}\vp(P)^{-1/6}P^{13/12}\\
 \lesssim  M^{3/2}\log^2 P\ \s(P)^{-1/2}\vp(P)^{1/2} P^{3/8}
 +  M^{7/6}\log^{6}P\ \s(P)^{-1/6}\vp(P)^{-1/6}P^{13/12}.
\end{multline}
In order to estimate the last two terms in \eqref{lastest} we will use the inequalities $x^{\g-\e_1}\lesssim_{\e_1}\vp(x)$, $\s(x)^{-1}\lesssim_{\e_2}x^{\e_2}$ which hold with arbitrary $\e_1, \e_2>0$.
Since $M=P^{1+\chi+\e}\vp(P)^{-1}$ with  $\chi>0$  such that $16(1-\g)+28\chi<1$ and $0<\e<\chi/100$, it is easy to see (taking $\e_1=\e_2=\e>0$ and $\log x\lesssim_{\e}x^{\e/50}$) that
\begin{align*}
  M^{3/2}\log^2 P\ \s(P)^{-1/2}\vp(P)^{1/2} P^{3/8}&=\left(P^{1+\chi+\e}\vp(P)^{-1}\right)^{3/2}\log^2 P\ \s(P)^{-1/2}\vp(P)^{1/2} P^{3/8}\\
  &=P^{15/8+3\chi/2+3\e/2}\vp(P)^{-1}\s(P)^{-1/2}\log^2 P\\
  &\lesssim P^{15/8+3\chi/2+4\e-\g}\lesssim P^{1-\chi+7/8+5\chi/2+4\e-\g}\lesssim P^{1-\chi-\e'},
\end{align*}
for some $\e'>0$, since $\log^2P\lesssim_{\e}P^{\e}$ and
\begin{align*}
 7/8+3\chi-\g<0\Longleftrightarrow 7+24\chi<8\g\Longleftrightarrow 8(1-\g)+24\chi<1.
\end{align*}
On the other hand, we get
\begin{align*}
  M^{7/6}\log^{6}P\ \s(P)^{-1/6}&\vp(P)^{-1/6}P^{13/12}
  =\left(P^{1+\chi+\e}\vp(P)^{-1}\right)^{7/6}\log^{6}P\ \s(P)^{-1/6}\vp(P)^{-1/6}P^{13/12}\\
  &=P^{27/12+7\chi/6+7\e/6}\vp(P)^{-8/6}\s(P)^{-1/6}\log^{6}P\\
  &\lesssim P^{27/12+7\chi/6+3\e-8\g/6}\lesssim P^{1-\chi+15/12+13\chi/6+3\e-8\g/6}\lesssim P^{1-\chi-\e'}.
\end{align*}
for some $\e'>0$, since $\log^{6}P\lesssim_{\e}P^{2\e/6}$ and
\begin{align*}
  15/12+14\chi/6-8\g/6<0\Longleftrightarrow 15+28\chi<16\g\Longleftrightarrow 16(1-\g)+28\chi<1.
\end{align*}
This provides the desired upper bound for \eqref{lastest} and the proof of Lemma \ref{formlem} is completed.
\end{proof}


\begin{thebibliography}{99}


\bibitem{BF} \textsc{A. Balog, J. P. Friedlander.}
\newblock {A hybrid of theorems of Vinogradov and Piatetski--Shapiro.}
\newblock {\textit{Pacific J. Math.} 156 (1992), 45--62.}

\bibitem{B} \textsc{J. Bourgain.}
\newblock {On $\Lambda(p)$ -- subsets of squares.}
\newblock {\textit{Israel J. Math.} 67 (1989), no. 3, 291--311.}

\bibitem{B1} \textsc{J. Bourgain.}
\newblock {On triples in arithmetic progression.}
\newblock {\textit{Geom. Funct. Anal.} 9 (1999), no. 5, 968--984.}

\bibitem{B2} \textsc{J. Bourgain.}
\newblock {Roth’s theorem on progressions revisited.}
\newblock {\textit{J. d'Analyse Math.} 104 (2008), no. 1, 155--192.}
	

\bibitem{CG} \textsc{D. Conlon, W. T. Gowers.}
\newblock {Combinatorial theorems in sparse random sets.}
\newblock {\textit{arXiv preprint arXiv:1011.4310.} (2010).}


\bibitem{GK} \textsc{S. W. Graham, G. Kolesnik.}
\newblock {Van der Corput’s method of exponential sums.}
\newblock {\textit{London Mathematical Society Lecture Note Series, 126, Cambridge University Press, Cambridge, (1991).}}

\bibitem{G}\textsc{B. Green.}
\newblock {Roth's theorem in the primes. }
\newblock {\textit{Ann. of Math.} 161 (2005), no. 3, 1609--1636.}

\bibitem{GT1}\textsc{B. Green, T. Tao.}
\newblock {Restriction theory of the Selberg sieve, with applications. }
\newblock {\textit{J. Th\'{e}or. Nombres Bordeaux,} 18 (2006), 147--182.}


\bibitem{GT}\textsc{B. Green, T. Tao.}
\newblock {The primes contain arbitrarily long arithmetic progressions. }
\newblock {\textit{Ann. of Math.} 167 (2008), no. 2,  481--–547.}


\bibitem{HL}\textsc{M. Hamel, I. {\L}aba.}
\newblock {Arithmetic structures in random sets. }
\newblock {\textit{Integers: Electronic Journal of Combinatorial Number Theory,}  8(A04), (2008), A04.}




\bibitem{HB}\textsc{D. R. Heath--Brown.}
\newblock {The Pjateckii--Sapiro prime number theorem. }
\newblock {\textit{J. Number Theory, 16 (1983), 242--266.}}

\bibitem{HB1}\textsc{D. R. Heath--Brown.}
\newblock {Integer sets containing no arithmetic progressions. }
\newblock {\textit{J. London Math. Soc.} 2 (1987), no. 3, 385--394.}
	
\bibitem{HdR}\textsc{H. A. Helfgott, A. de Roton.}
\newblock {Improving Roth's theorem in the primes. }
\newblock {\textit{Int. Math. Res. Notices. } 2011 (2011), no. 4, 767--783.}






\bibitem{IK} \textsc{H. Iwaniec, E. Kowalski.}
\newblock {Analytic Number Theory.}
\newblock {\textit{Vol. 53, Amer. Math. Soc. Colloquium
Publications, Providence RI, (2004).}}

\bibitem{KLR}\textsc{Y. Kohayakawa, T. {\L}uczak, V. R\"{o}dl.}
\newblock {Arithmetic progressions of length three in subsets of a random set.}
\newblock {\textit{Acta Arith.}  75(2), (1996), 133--163.}

\bibitem{Kol}
\textsc{G. Kolesnik.}
\newblock {The distribution of primes in sequences of the form $\lfloor n^c\rfloor$.}
\newblock { \textit{Mat. Zametki}, 2 (1967), 117--128.}

\bibitem{Kol1}
\textsc{G. Kolesnik.}
\newblock { Primes of the form $\lfloor n^c\rfloor$.}
\newblock { \textit{Pacific J. Math.}, 118 (1985), 437--447.}


\bibitem{Kum}
\textsc{A. Kumchev.}
\newblock { On the Piatetski--Shapiro--Vinogradov theorem.}
\newblock { \textit{J. Th\'{e}or. Nombres Bordeaux,}, 9 (1997), no. 1, 11--23.}



\bibitem{Leit}
\textsc{D. Leitmann.}
\newblock { The distribution of prime numbers in sequences of the form $\lfloor f(n)\rfloor$.}
\newblock { \textit{Proc. London Math. Soc.}, 35 (1977), no. 3, 448--462.}

\bibitem{LR}
\textsc{H.Q. Liu, J.Rivat.}
\newblock { On the Piateski--Shapiro prime number theorem.}
\newblock { \textit{Bull. London Math. Soc.}, 24 (1992), 143--147.}




   \bibitem{M} \textsc{M. Mirek.}
\newblock {$\ell^p(\Z)$ -- boundedness of discrete maximal functions along thin subsets of primes and pointwise ergodic theorems.}
\newblock {\textit{Preprint,} (2013).}

   \bibitem{MS} \textsc{G. Mockenhaupt, W. Schlag.}
\newblock {On the Hardy–Littlewood majorant problem for random sets.}
\newblock {\textit{J. Funct. Anal.} 256 (2009), no. 4, 1189--1237.}

\bibitem{Nas} \textsc{E. Naslund.}
\newblock {On Improving Roth's Theorem in the Primes.}
\newblock {\textit{Preprint,} (2013).}

\bibitem{Nat} \textsc{M. B. Nathanson.}
\newblock {Additive Number Theory. The Classical Bases.}
\newblock {\textit{Springer--Verlag, (1996).}}


\bibitem{PS} \textsc{I. Piatetski--Shapiro.}
\newblock {On the distribution of prime numbers in sequences of the form $\lfloor f(n)\rfloor$.}
\newblock {\textit{Math. Sbornik} 33 (1953), 559–-566.}



\bibitem{RS} \textsc{J. Rivat,  P. Sargos.}
\newblock {Nombres premiers de la forme $\lfloor n^c\rfloor$.}
\newblock {\textit{Canad. J. Math.} 53 (2001), 414--433.}

\bibitem{Rot} \textsc{K. F. Roth.}
\newblock {On certain sets of integers.}
\newblock {\textit{J. London Math. Soc.} 28 (1953), 104--109.}

   \bibitem{San1} \textsc{T. Sanders.}
\newblock {On certain other sets of integers.}
\newblock {\textit{J. Anal. Math.} 116 (2012), 53–--82.}

   \bibitem{San2} \textsc{T. Sanders.}
\newblock {On Roth's theorem on progressions.}
\newblock {\textit{Ann. of Math. } (2) 174 (2011), no. 1, 619--636}


   \bibitem{Ser} \textsc{E. Szemer\'{e}di.}
\newblock {Integer sets containing no arithmetic progressions.}
\newblock {\textit{Acta Math. Hungar.} 56 (1990), no. 1, 155--158.}

  \bibitem{Ser1} \textsc{E. Szemer\'{e}di.}
\newblock {On sets of integers containing no $k$ elements in arithmetic progression.}
\newblock {\textit{ Acta Arith.} 27(585), (1975), 199--245.}


\bibitem{TV} \textsc{T. Tao, V. Vu.}
\newblock {Additive combinatorics.}
\newblock {\textit{Cambridge University Press, vol. 105, (2006).}}


\bibitem{vCor} \textsc{J. G. Van der Corput. }
\newblock {\"{U}ber Summen von Primzahlen und Primzahlquadraten.}
\newblock {\textit{Math. Ann, 116 (1939), no. 1, 1--50.}}

\bibitem{Var} \textsc{P. Varnavides. }
\newblock {On certain sets of positive density.}
\newblock {\textit{J. London Math. Soc. 1 (1959), no. 3, 358--360.}}
	


\end{thebibliography}
\end{document}